 \documentclass[final,onefignum,onetabnum]{siamart171218}

\ifpdf
\hypersetup{
  pdftitle={Transmission and Reflection coefficients for Schr{\"o}dinger Operators with Truncated Periodic Potentials that support gap states},
  pdfauthor={J.C. Stellman and J.L. Marzuola}
}
\fi

\title{Transmission and Reflection coefficients for Schr{\"o}dinger Operators with Truncated Periodic Potentials that support defect states}
\author{J.C. Stellman\thanks{Undergraduate at Department of Mathematics, University of North Carolina, Chapel Hill, NC during the production of this work}
\and Jeremy L. Marzuola\thanks{Department of Mathematics, University of North Carolina, Chapel Hill, NC (\email{marzuola@email.unc.edu}).}}

\newsiamthm{assn}{Assumption}
\newsiamremark{rmk}{Remark}
\newsiamthm{prop}{Proposition}
\newsiamthm{result}{Result}

\newcommand{\re}{\operatorname{Re}}
\newcommand{\im}{\operatorname{Im}}
\newcommand{\dz}{\partial_z}
\newcommand{\dw}{\partial_w}

\newcommand{\zy}{z_{\hat Y}}
\newcommand{\wy}{w_{\hat Y}}
\newcommand{\zx}{z_{\hat X}}

\newcommand{\zetay}{\zeta_{\hat Y}}

\newcommand{\zyn}{z_{\hat Y_n}}
\newcommand{\wyn}{w_{\hat Y_n}}

\newcommand{\zetayn}{\zeta_{\hat Y_n}}

\newcommand{\uz}{u_{\mathbf{\zeta}}}
\newcommand{\ue}{u_{\mathbf{\eta}}}
\newcommand{\vz}{v_{\mathbf{\zeta}}}
\newcommand{\ve}{v_{\mathbf{\eta}}}

\newcommand{\uzn}{u_{\mathbf{\zeta}_n}}
\newcommand{\vzn}{v_{\mathbf{\zeta}_n}}
\newcommand{\uen}{u_{\mathbf{\eta}_n}}
\newcommand{\ven}{v_{\mathbf{\eta}_n}}

\newcommand{\vecu}{\vec{u}_\zeta}
\newcommand{\Ue}{U_\eta}

\newcommand{\dztpe}{\partial_z\Theta^+(\eta)}
\newcommand{\dztme}{\partial_z\Theta^-(\eta)}
\newcommand{\dwtpe}{\partial_w\Theta^+(\eta)}
\newcommand{\dwtme}{\partial_w\Theta^-(\eta)}
\newcommand{\dztpz}{\partial_z\Theta^+(\zeta)}
\newcommand{\dztmz}{\partial_z\Theta^-(\zeta)}
\newcommand{\dwtpz}{\partial_w\Theta^+(\zeta)}
\newcommand{\dwtmz}{\partial_w\Theta^-(\zeta)}

\newcommand{\dztpen}{\partial_z\Theta^+(\eta_n)}
\newcommand{\dztmen}{\partial_z\Theta^-(\eta_n)}
\newcommand{\dwtpen}{\partial_w\Theta^+(\eta_n)}
\newcommand{\dwtmen}{\partial_w\Theta^-(\eta_n)}

\usepackage{amsmath, amssymb, graphicx}

\begin{document}

\maketitle

\begin{abstract}
  We consider scattering waves through truncated periodic potentials with perturbations that support localized gap eigenstates. In a small complex neighborhood around an assumed positive bound state of the model operator, we prove the existence of a distinct zero reflection state, or transmission resonance. We compare its location to a previously found scattering resonance and use the properties of solutions near these interesting points to analyze the behavior of transmission and reflection coefficients of scattering solutions near the assumed bound state.  By example, we also discuss the truncated simple harmonic oscillator and compare the analysis to the crystalline case.
\end{abstract}

\tableofcontents

\section{Introduction}
\label{sec:introduction}

 We study transmission and reflection for one-dimensional scattering potentials obtained by truncating full-line operators with well-understood spectral structure. Our motivation is to identify the scattering signature of a localized eigenstate when an infinite medium is replaced by a finite sample.  A common strategy for creating a structure that hosts a localized bound state is to build a structure which is perfectly periodic except for a defect region or using trapping by a well-like potential that grows sufficiently fast at infinity. Bound states created by defects in an otherwise periodic structure are known as \emph{defect states}.  Each of these settings allows for isolated, discrete spectrum to form.  In the settings of perturbation of periodic media, many previous works have considered the existence of defect states across a large variety of models; see e.g. \cite{1976Simon_2,1986DeiftHempel,1993GesztesySimon,1997FigotinKlein,Figotin-Klein:98,Cherdantsev2009,Hoefer-Weinstein:11,2011BronskiRapti,duchene2015oscillatory,DVW:15,Kamotski2018a} and the numerical works \cite{2001FigotinGoren,2005Soussi}.

In higher dimensional lattice structures, the study of "topological edge states" has become a very important topic.  These are uni-directional states that decay away from the physical edge of a material, or off of a line of defects within a material, and have been studied considering wave-guiding and lasing applications.  See \cite{harari2018topological,bandres2018topological}, as well as \cite{fefferman_leethorp_weinstein_memoirs,2017FeffermanLee-ThorpWeinstein_2,2018FeffermanWeinstein,2019Drouot_2,2018DrouotFeffermanWeinstein_pre,2019Lee-ThorpWeinsteinZhu,2019DrouotWeinstein,2020Drouot}. For related numerical work, see \cite{2018ThickeWatsonLu}. Our results do not directly establish a higher-dimensional topological scattering theorem. A possible connection is through partial Fourier or Bloch decomposition along a straight interface: when translation symmetry reduces the transverse problem to a family of one-dimensional operators, one may study finite-width truncations of individual fibers. Controlling the resulting scattering coefficients uniformly in the wavenumber along the interface, and handling band crossings, would require additional analysis beyond the present work.

 In a related previous work, the authors of \cite{LMW22} studied how defect eigenstates become scattering resonances after truncation of an infinite periodic structure.  Our goal here is to study a complementary problem related to the transmission/reflection coefficients associated to the scattering matrix for the truncated structure.  The connection between scattering resonances and transmission resonances will be made clear below.  Initially, the transmission/reflection coefficient problem may seem perpendicular since resonances can easily be seen as generalizations of the notion of eigenvalue, while the transmission/reflection coefficients are known to be crucially linked to the essential spectrum of the infinite structure that has been truncated.  However, near the energy of a defect state, we will particularly observe that for a large class of examples the transmission resonance leads to non-uniform convergence of transmission/reflection coefficients to that of the spectrum of the relevant infinite spectral structure Schr\"odinger operators, which has been explored in other examples such as \cite{duchene2011scattering,osting2013long,duchene2014scattering}.  
 
The earlier resonance calculation of \cite{LMW22} locates a pole of the meromorphically continued scattering matrix and quantifies the lifetime of a truncated defect mode. A pole location alone does not determine the real-energy transmission probability or the location of a reflection zero. Here we study the reflection numerator as a separate analytic object, compare its zero with the pole, and provides estimates on its behavior nearby. The narrow peaks in the numerical figures illustrate why pointwise scattering limits can fail to be uniform in energy near a defect eigenvalue. In asymmetric settings the peak height also depends on endpoint phase and imbalance, not merely on the existence of a long-lived resonance.\par

 Our techniques arise from similar methods used for comparison of scattering resonances to eigenvalues as pursued in \cite{LMW22} (and are heavily influenced by the truncated Simple Harmonic Oscillator example in \cite{DZ19}).  Indeed, to streamline exposition our main results will focus on truncation of periodic structures, but we will demonstrate below in Section \ref{sec:tho} that our results naturally extend to truncation of trapping potentials as well using the simple harmonic oscillator as an example. Scattering resonance by comparison to infinite barriers has also been studied in \cite{dobson2013resonances}.  Recently, nonlinear implications for both transmission and scattering resonances were explored in \cite{turner2025resonance}.

To precisely state our results, we consider the one-dimensional Schr{\"o}dinger operator
\begin{equation}
\label{eq:h}
    H := D_x^2+V(x) \qquad D_x := -i\partial_x,
\end{equation}
where the potential, $V$, is a real function. We make the following assumptions on the regularity of $V$.

\begin{assn}[Regularity and translation symmetry of $V$]
\label{assn:regularity}
    We assume that $V$ is smooth: $V\in C^\infty(\mathbb R)$. Then, we assume that $V$ can be written as
    \begin{equation}
    \label{eq:vform}
        V(x) = V_{\text{per}}(x) + V_{\text{def}}(x),
    \end{equation}
    where $V_{\text{per}}$ is periodic, i.e. $V_{\text{per}}(x+1) = V_{\text{per}}(x)$ for all $x\in\mathbb R$, and $V_\text{def}(x)$ is compactly supported, i.e. there exists $\rho\geq0$ such that 
    \begin{equation}
    \label{eq:compact}
        |x|>\rho \Longrightarrow V_\text{def}(x)=0.
    \end{equation}
\end{assn}

Now suppose that $H$ has a bound state with positive energy, i.e., that there exist $E>0$ and $\Phi(x)\in L^2(\mathbb R)$ such that 
\begin{equation}
\label{eq:boundstate}
    H\Phi(x) = E\Phi(x).
\end{equation}
We assume throughout the analysis that $E$ lies in an open spectral gap of the background periodic operator $D_x^2+V_{\mathrm{per}}$. In this gap, the Floquet multipliers of the one-period transfer matrix are real and reciprocal, with moduli different from one. Thus on each periodic tail there is a decaying Floquet solution and a growing Floquet solution. Matching the square-integrable eigenfunction to the decaying direction at both ends yields the exponential decay estimates below. Consideration of the independent solution yields the exponential growth estimate. This is the general mechanism behind the bounds used here, without requiring the reader to fully reconstruct the Floquet argument from \cite{LMW22}.

Our goal is to prove that when the structure modeled by $H$ is truncated sufficiently far from $x=0$, a state of zero reflection, or transmission resonance, exists uniquely very nearby $E$. We then compare the location of this state to a previously found scattering resonance, which we denote as $\zx$ (described in \cite{LMW22}), and analyze reflection and transmission nearby. Specifically, we aim to prove that the operator
\begin{equation}
\label{eq:htrunc}
    H_{\text{trunc}} := D_x^2 + V_{\text{trunc}}(x), 
\end{equation}
where
\begin{equation}
\label{eq:vtrunc}
    V_{\text{trunc}}(x) = \begin{cases} V(x) & |x|\leq M \\ 0 & |x|>M \end{cases}
\end{equation}
acting on $L^2(\mathbb R)$ has a zero reflection state $\zy$ with $\im\zy<0$ nearby to $E$ in the complex plane for $M$ sufficiently large (such that $M>\rho$). To see this we define a set of fundamental solutions and use their growth and decay to describe the location of the zero reflection state and the behavior of transmission and reflection around this state.

\begin{rmk}
\label{rmk:Vpm}
    Assumption \ref{assn:regularity} imposes a single periodic background and a compactly supported defect. A related setting has two limiting periodic backgrounds, $V_{\mathrm{per},-}$ and $V_{\mathrm{per},+}$, with $V(x)-V_{\mathrm{per},\pm}(x)$ decaying as $x\to\pm\infty$. The Floquet and variation-of-constants estimates admit extensions to this setting when the relevant energy belongs to a common open gap; quantitative real-transmission conclusions additionally require the endpoint phase nondegeneracy stated later. We do not claim that the principal theorem applies verbatim. In the numerical section we illustrate the expected scattering behavior using such a two-sided asymptotically periodic potential, without treating the plots as a certified application of the main theorem.
\end{rmk}

When $\Phi(0)\neq0$ (after multiplying by a constant we can assume $\Phi(0)=1$), define $\uz\in C^\infty([-M,M])$ and $\vz\in C^\infty([-M,M])$ as the solutions of
\begin{equation}
\begin{aligned}
\label{eq:fundamentals1}
    &(D_x^2+V(x)-z)\uz = 0, \quad \uz(0)=1, \quad \uz'(0) = w \\
    &(D_x^2+V(x)-z)\vz = 0, \quad \vz(0)=\frac{-w}{1+w^2}, \quad \vz'(0) = \frac{1}{1+w^2},
\end{aligned}
\end{equation}
for parameters $\zeta = (w,z)\in\mathbb C^2$ and $\eta = (\Phi'(0),E)$. The value $w=\Phi'(0)$ is fixed at the distinguished bound state; varying $w$ nearby is an auxiliary device for the analytic implicit-function argument we use below. With these definitions $\uz$ and $\vz$ form a fundamental solution set with Wronskian $1$ and $\ue(x)=\Phi(x)$ for $|x|\leq M$. When $\Phi(0)=0$ (note after multiplying by a constant we can assume $\Phi'(0)=1$), define $\uz(x)$ and $\vz(x)$ by 
\begin{equation}
\begin{aligned}
\label{eq:fundamentals2}
    &(D_x^2+V(x)-z)\uz = 0, \quad \uz(0)=-w, \quad \uz'(0) = 1 \\
    &(D_x^2+V(x)-z)\vz = 0, \quad \vz(0)=\frac{-1}{1+w^2}, \quad \vz'(0) = \frac{-w}{1+w^2},
\end{aligned}
\end{equation}
and let $\eta = (0,E)$. Again, $\uz$ and $\vz$ form a fundamental solution set and $\ue(x)=\Phi(x)$ for $|x|\leq M$. Because $\Phi\in L^2(\mathbb R)$ and $E$ lies in a spectral gap of the periodic tails, Floquet theory selects the stable direction on the right and the opposite stable direction on the left; the same global eigenfunction therefore decays at both ends. The Wronskian-normalized independent solution has a nonzero growing component on each tail, since two linearly independent solutions cannot both decay on the same half-line. See the gap-Floquet arguments in Sections 5 and 9 of \cite{LMW22}. In the two-background setting the same reasoning uses the stable subspace of each limiting periodic operator separately, provided $E$ belongs to both gaps. Thus $\ue$ decays exponentially and $\ve$ grows at most exponentially: That is, there exist constants $C>0$ and $k>0$ such that
\begin{equation}
\label{eq:uedecay}
    |\ue(x)| \leq Ce^{-k|x|}, \quad |\ue'(x)| \leq Ce^{-k|x|},
\end{equation}
and 
\begin{equation}
\label{eq:vegrowth}
    |\ve(x)| \leq Ce^{k|x|}, \quad |\ve'(x)| \leq Ce^{k|x|}.
\end{equation}

Also note that we can consider a scattering framework for solutions. That is, we can consider distorted plane waves $u(x)$ at complex energy $z$ which satisfy
\begin{equation}
    H_{\text{trunc}}u = zu
\end{equation}
for $H_{\text{trunc}}$ given by \cref{eq:htrunc} subject to the boundary conditions
\begin{equation}
\label{eq:scatteringsolutions}
    u(x) = \begin{cases}
        e^{i\sqrt{z}\,x} + R_- (z)e^{-i\sqrt{z}\,x}, & x \leq -M \\
        T(z)e^{i\sqrt{z}\,x}, & x \geq M
    \end{cases},
\end{equation}
where $R_-(z)$ and $T(z)$ are meromorphic functions which we call the reflection and transmission coefficients, respectively.  There is a similar set-up for $R_+$ where the incoming wave is from $+\infty$. We will give more details about the structure of $R_\pm$ and $T$ as meromorphic functions in Section \ref{sec:comparison} below, but for now, we simply state that there exist analytic functions $\hat{X}(z)$ and $\hat{Y} (z)$ such that $R_\pm(z) = \frac{\hat{Y}(\mp z)}{\hat{X}(z)}$ and hence we must determine for which values of $z$ the function $\hat{Y}(\pm z) = 0$.  We will focus on the case $R(z) := R_-(z)$ for simplicity, though there is a clear modification for $R_+(z)$.  Thus, we define a zero-reflection state, or transmission resonance, as a value $z$ such that ${\rm Im} z < 0$ and $R(z) = 0$.  This gives explicit behavior of the solution on the left and right, which very much mimics the construction of a scattering resonance.  Indeed, the construction in \cite{LMW22} constructs a $z$ such that $\hat{X}(z) = 0$.

Using these fundamental solutions and the scattering framework, we can introduce the main results:

\begin{theorem}
\label{result:zeroReflection}
Let $V(x)$ satisfy \cref{assn:regularity}, let $\rho>0$ as in \cref{eq:compact}, and $\Phi$ be a bound state with eigenvalue $E>0$ as in \cref{eq:boundstate} in an open spectral gap of $D_x^2 + V_\text{per}$. Then there exists $k>0$ and $M_0>\rho\geq0$ such that for all $M\geq M_0$, there exists a unique zero reflection state (or transmission resonance) $\zy$ in an exponentially small neighborhood about the bound state. That is,
\begin{equation}
    \zy\in\Big\{ z : |z-E| \leq \frac{1}{M^2}e^{-kM} \Big\}.
\end{equation}
\end{theorem}
A precise statement and proof of \cref{result:zeroReflection} is given with \Cref{thm:zeroreflection}.

\begin{theorem}
\label{result:trBounds}
Let $V(x)$ satisfy \cref{assn:regularity}, let $R(z)$ be as in \cref{eq:scatteringsolutions}, let $\rho>0$ as in \cref{eq:compact}, and $\Phi$ be a bound state with eigenvalue $E>0$ as in \cref{eq:boundstate} in an open spectral gap of $D_x^2 + V_\text{per}$. Let $\ue$ and $\ve$ be as in \cref{eq:fundamentals1,eq:fundamentals2} and let $k>0$ and $M_0$ as in \cref{result:zeroReflection}. Then for all $M\geq M_0$, whenever
\begin{equation}
\label{eq:distancequotient}
    \Bigg|\frac{[\ve'(M)]^2 + E[\ve(M)]^2}{[\ve'(-M)]^2 + E[\ve(-M)]^2}\Bigg| < 2
\end{equation}
there a exists a neighborhood $\Gamma_M$ of $\zy$ with 
\begin{equation}
\label{eq:gamma}
    \Gamma_M = \Big\{ z\in\mathbb{C} : |z-\zy| \leq |\im(\zx-\zy)|\Big\},
\end{equation}
where $R(z)$ is analytic in $\Gamma_M$ and the following estimates hold:
\begin{equation}
\label{eq:reflectionbound}
    \sup_{z\in\Gamma_M}|R(z)| \leq Ce^{-kM},
\end{equation}
and
\begin{equation}
\label{eq:dreflectiondbound}
\sup_{z\in\Gamma_M}|\partial_zR(z)| \leq Ce^{kM}.
\end{equation}
\end{theorem}
A precise statement and proof of \cref{result:trBounds} is given with \Cref{thm:reflectionbounds}.

\paragraph{Physical interpretation of the two resonance notions}
A scattering resonance is a pole of the meromorphically continued scattering matrix: its imaginary part encodes the decay rate of a metastable defect mode and therefore its lifetime. By contrast, a reflection zero is a cancellation of the reflected amplitude. When that zero lies at a real positive energy, conservation of flux forces unit transmission for the corresponding one-dimensional real potential. For a complex reflection zero, however, this physical conclusion does not follow without further control of the real-energy scattering coefficients. Thus the two objects answer distinct questions relevant to a finite resonant sample: how long an excited mode persists, and how efficiently a wave at a given real frequency passes through the sample. These distinctions are useful when considering resonance tuning.

\subsection*{Acknowledgments}  The authors thank Alex Watson, Jianfeng Lu, Michael Weinstein and Mikael Rechtsman for helpful conversations throughout the development of this result.  We also wish to thank Kiril Datchev and Tanya Christiansen for helpful comments on an early version of the result. J.L.M. acknowledges support from NSF grant DMS-2307384. J.S. was supported by NSF RTG DMS-2135998 and the SMART Scholarship from the Department of Defense.  ChatGPT version 6 was used in the final revision of the document and assisted the authors in establishing Theorem 5.4.

\section{States of Zero Reflection Near Eigenvalues of $H$}
\label{sec:reflectionPoint}
Recall our fundamental solutions defined in \cref{eq:fundamentals1,eq:fundamentals2}. We define
\begin{equation}
\label{eq:xmap}
    (X_1^\pm(\zeta),X_2^\pm(\zeta)) := (\uz(\pm M), \uz'(\pm M)),
\end{equation}
and let
\begin{equation}
\label{eq:thetamap}
    \Theta^\pm(\zeta):= X_2^\pm(\zeta) - i\sqrt{z}X_1^\pm(\zeta).
\end{equation}
Then we have that when $E>0$, $\zy$ is a zero reflection point of $H_\text{trunc}$ if and only if
\begin{equation}
\label{eq:Thetamap}
    \mathbf{\Theta}(\zetay)=0, \text{ where } \mathbf{\Theta}(\zeta) := \begin{pmatrix} \Theta^+(\zeta) \\ \Theta^-(\zeta) \end{pmatrix},
\end{equation}
where $\zetay = (\wy,\zy)$ for some $\wy\in\mathbb C$. We can now introduce the main theorem for this section, which is a precise statement of \cref{result:zeroReflection}:

\begin{theorem}
\label{thm:zeroreflection}
    Let $V(x)$ satisfy \cref{assn:regularity}, let $\rho>0$ be as in \cref{eq:compact}, and let $\Phi$ be a bound state with eigenvalue $E>0$ as in \cref{eq:boundstate} in an open spectral gap of $D_x^2 + V_\text{per}$. Then there exists $k>0$ and an $M_0>\rho\geq0$ such that for all $M\geq M_0$,
    \begin{itemize}
        \item[(1)] $\mathbf{\Theta}(\zeta)$ given by \cref{eq:Thetamap} has a unique zero $\zetay$ in the ball
        \begin{equation}
        \label{eq:bally}
            \Omega_M := \Big\{ \zeta : |\zeta-\eta| \leq \frac{1}{M^2}e^{-kM} \Big\}.
        \end{equation}
        \item[(2)] The location of $\zetay$ can be characterized as 
        \begin{equation}
        \label{eq:reflectionlocation}
            \zetay = \eta - \Xi\mathbf{\Theta}(\eta) + O(e^{-4kM}),
        \end{equation}
        where $\Xi$ is the matrix
        \begin{equation}
        \label{eq:ximatrix}
            \Xi := \frac{1}{\mathcal{N}(\eta)} \begin{pmatrix} \dztme & -\dztpe \\ -\dwtme & \dwtpe \end{pmatrix},
        \end{equation}
        where $\mathcal{N}(\eta) := \dwtpe\dztme - \dztpe\dwtme$.
    \end{itemize}
\end{theorem}

The overall idea of the proof of \cref{thm:zeroreflection} is to construct a map so that a fixed point of this map is a point of zero reflection. We then prove that this map is a contraction and thus has a unique fixed point in $\Omega_M$.  We provide a sketch of the proof before introducing the necessary lemmas to prove \cref{thm:zeroreflection}. Define $\Psi:\mathbb C^2\rightarrow \mathbb C^2$ by
\begin{equation}
\label{eq:psimap}
    \Psi(\zeta) := \zeta - \Xi\mathbf{\Theta}(\zeta),
\end{equation}
where $\Xi$ is defined by \cref{eq:ximatrix}. Remarking that $\mathcal{N}(\eta)\neq0$ (discussed later), then $\det\Xi=[\mathcal{N}(\eta)]^{-3}$ so that $\Xi$ is invertible, and hence
\begin{equation}
\label{eq:psitheta}
    \Psi(\zeta)=\zeta \Longleftrightarrow \mathbf{\Theta}(\zeta)=0.
\end{equation}
It remains to show that $\Psi$ is a contraction map in the ball $\Omega_M$ given by \cref{eq:bally}, where $k$ describes the exponential decay of the bound state. Parts one and two of \cref{thm:zeroreflection} then follow from the Banach fixed point theorem and the asymptotic formula for the fixed point as $\lim_{n\to\infty}\Psi^n(\eta)$ respectively. 

\begin{rmk}
\label{rmk:inverse}
    The matrix $\Xi$ is the inverse of 
    \begin{equation}
    \label{eq:xiinverse}
        \begin{pmatrix} \dwtpe & \dztpe \\ \dwtme & \dztme \end{pmatrix},
    \end{equation}
    which is the Jacobian of the map 
    \begin{equation}
    \label{eq:thetamapalt}
        \zeta \mapsto (\Theta^+(\zeta), \Theta^-(\zeta))
    \end{equation}
    evaluated at $\zeta=\eta$. This is analogous to a proof of resonance locations for parity-symmetric potentials presented in \cite{LMW22}. 
\end{rmk}

\begin{corollary}
\label{cor:reflectionasymptotics}
    Let $\zetay$ be as in \cref{eq:reflectionlocation}. Then
    \begin{equation}
    \begin{aligned}
    \label{eq:wylocation}
        \wy = w_0 &- \Big[\int_{-M}^0 \ue^2(y)\,dy \big(\ve'(-M)-i\sqrt{E}\ve(-M)\big) \big(\ue'(M)-i\sqrt{E}\ue(M)\big) \\ 
        &+ \int_0^M \ue^2(y)\,dy \big(\ve'(M)-i\sqrt(E)\ve(M)\big) \big(\ue'(-M)-i\sqrt(E)\ue(-M)\big)\Big] \\
        &\times \Big[\int_{-M}^M \ue^2(y)\,dy \big(\ve'(M)-i\sqrt{E}\ve(M)\big) \big(\ve'(-M)-i\sqrt{E}\ve(-M)\big)\Big]^{-1} \\ 
        &+ O(e^{-4kM}),
    \end{aligned}
    \end{equation}
    where $w_0=\Phi'(0)$ when $\Phi(0)=0$ and $w_0=0$ when $\Phi(0)\neq 0$, and
    \begin{equation}
    \begin{aligned}
    \label{eq:zylocation}
        \zy = E &- \Big[\big(\ve'(M)-i\sqrt{E}\ve(M)\big) \big(\ue'(-M)-i\sqrt{E}\ue(-M)\big) \\
        &- \big(\ve'(-M)-i\sqrt{E}\ve(-M)\big) \big(\ue'(M)-i\sqrt{E}\ue(M)\big)\Big] \\
        &\times \Big[\int_{-M}^M \ue^2(y)\,dy \big(\ve'(M)-i\sqrt{E}\ve(M)\big) \big(\ve'(-M)-i\sqrt{E}\ve(-M)\big)\Big]^{-1} \\ 
        &+ O(e^{-4kM}).
    \end{aligned}
    \end{equation}
\end{corollary}
\begin{proof}
    See \Cref{sec:reflectionasymptoticsproof}.
\end{proof}

\section{Proof of \Cref{thm:zeroreflection}}
\label{sec:reflectionproof}

We will first show that the form of $\Xi$ makes $\Psi$ a contraction in the ball $\Omega_M$. We write $\Psi = (\Psi_1,\Psi_2)^T$. Then we have that $\Psi$ is a contraction if every element of the Jacobian matrix
\begin{equation}
\label{eq:jacobian}
    J:= \begin{pmatrix} J_{11} & J_{12} \\ J_{21} & J_{22} \end{pmatrix} := \begin{pmatrix} \partial_w\Psi_1 & \partial_z\Psi_1 \\ \partial_w\Psi_2 & \partial_z\Psi_2 \end{pmatrix}
\end{equation}
can be bounded uniformly by $\frac{1}{2}$ in $\Omega_M$. Recalling the form of $\Psi$ \cref{eq:psimap} and $\Xi$ \cref{eq:ximatrix}, and assuming that $\Xi$ is independent of $\zeta$, we have the diagonal entries of $J$ given by
\begin{equation}
\begin{aligned}
\label{eq:jacobiandiagonals}
    J_{11} = 1 - \frac{\dwtpz\dztme - \dztpe\dwtme}{\mathcal{N}(\eta)}, \\
    J_{22} = 1 - \frac{\dwtpe\dztmz - \dztpz\dwtme}{\mathcal{N}(\eta)},
\end{aligned}
\end{equation}
while the off-diagonal terms are
\begin{equation}
\begin{aligned}
\label{eq:jacobianoffdiagonals}
    J_{12} = -\frac{\dztpz\dztme - \dztpe\dztmz}{\mathcal{N}(\eta)}, \\
    J_{21} = -\frac{\dwtpz\dwtme - \dwtpe\dwtmz}{\mathcal{N}(\eta)}.
\end{aligned}
\end{equation}
To see that this choice of $\Xi$ makes the entries of $J$ small for $\zeta\in\Omega_M$ and $M$ sufficiently large, we rewrite the diagonal terms of $J$ as
\begin{equation}
\begin{aligned}
\label{eq:jacobiandiagnoalsrewrite}
    J_{11} = \frac{1}{\mathcal{N}(\eta)}\big[\big(\dwtpe-\dwtpz\big)\dztme - \dztpe\big(\dwtme-\dwtmz\big)\big], \\
    J_{22} = \frac{1}{\mathcal{N}(\eta)}\big[\dwtpe\big(\dztme-\dztmz\big) - \big(\dztpe-\dztpz\big)\dwtme\big]
\end{aligned}
\end{equation}
and the off-diagonal terms of $J$ as 
\begin{equation}
\begin{aligned}
\label{eq:jacobianoffdiagonalsrewrite}
    J_{12} = \frac{1}{\mathcal{N}(\eta)}\big[\dztpe\big(\dztmz-\dztme\big) + \big(\dztpe-\dztpz\big)\dztme\big], \\
    J_{21} = \frac{1}{\mathcal{N}(\eta)}\big[\dwtpe\big(\dwtmz-\dwtme\big) - \big(\dwtpz-\dwtpe\big)\dwtme\big].
\end{aligned}
\end{equation}
We see that every component of $J$ is a sum of two terms, and we will show that each component can be bounded by $\frac{1}{2}$ by showing that each of these terms can be bounded by $\frac{1}{4}$. We give an example simplification of these terms before making the necessary estimates precise. For example, we aim to prove that
\begin{equation}
\label{eq:termbound}
    \Big|\frac{1}{\mathcal{N}(\eta)}\big(\dztpe-\dztpz\big)\dwtme\Big| < \frac{1}{4},
\end{equation}
or equivalently
\begin{equation}
\label{eq:termbound2}
    |\dztpe-\dztpz|\,|\dwtme| < \frac{1}{4}|\mathcal{N}(\eta)|
\end{equation}
for all $\zeta\in\Omega_M$. By an application of the mean value theorem, we have that
\begin{equation}
\label{eq:termboundmvt}
    |\dztpe-\dztpz|\,|\dwtme| \leq |\zeta-\eta|\sup_{\zeta\in\Omega_M}|\partial_z^2\Theta^+(\zeta)|
\end{equation}
for $\zeta\in\Omega_M$. Since $|\zeta-\eta|$ is exponentially small in M for $\zeta\in\Omega_M$, estimate \cref{eq:termbound} is proved if we can show that
\begin{equation}
\label{eq:exestimates}
    \sup_{\zeta\in\Omega_M}|\partial_z^2\Theta^+(\zeta)|\leq Ce^{kM}, |\dwtme|\leq Ce^{kM}, \text{ and } |\mathcal{N}(\eta)|\geq Ce^{kM}
\end{equation}
for constants $k>0$ and $C>0$. It is clear that with bounds similar to \cref{eq:exestimates} the other terms in \cref{eq:jacobiandiagnoalsrewrite,eq:jacobianoffdiagonalsrewrite} can be bounded along the same lines.

Now that the basic idea has been discussed, we present and prove the lemmas necessary to show that $\Psi$ is a contraction.

\begin{lemma}
\label{lem:detbound}
    Let $\Theta^\pm(\zeta)$ be as in \cref{eq:thetamap} and $k$ the decay factor associated to $V$ as given in \eqref{eq:uedecay}-\eqref{eq:vegrowth}. There exist constants $C>0$ and $M_0>\rho>0$ such that for all $M\geq M_0$, 
    \begin{equation}
    \label{eq:detdef}
        \mathcal{N}(\eta):= \dwtpe\dztme - \dztpe\dwtme
    \end{equation}
    satisfies
    \begin{equation}
    \label{eq:detbound}
        |\mathcal{N}(\eta)| \geq Ce^{2kM}.
    \end{equation}
\end{lemma}

\begin{lemma}
\label{lem:thetaestimates}
    Let $k$ be as in \cref{lem:detbound} and $\Omega_M$ as in \cref{eq:bally}. Then there exist positive constants $C$ and $M_0>\rho>0$ such that for all $M\geq M_0$ the following estimates hold
    \begin{equation}
    \label{eq:firstestimates}
        |\partial_z\Theta^\pm(\eta)| \leq C e^{kM}, |\partial_w\Theta^\pm(\eta)| \leq C e^{kM}
    \end{equation}
    and
    \begin{equation}
    \label{eq:secondestimates}
\sup_{\zeta\in\Omega_M}|\partial_z^2\Theta^\pm(\zeta)|\leq C e^{kM}, \sup_{\zeta\in\Omega_M}|\partial_w^2\Theta^\pm(\zeta)|\leq C e^{kM}.
    \end{equation}
\end{lemma}

\subsection{Preliminaries for Proofs of \Cref{lem:detbound,lem:thetaestimates}}
\label{sub:preliminaries}
Using the definitions of $\Theta^\pm(\zeta)$ given by \cref{eq:xmap,eq:thetamap} we have that
\begin{equation}
\begin{aligned}
\label{eq:thetaderivatives}
    &\dz\Theta^\pm(\eta) = \dz\uz'(\pm M)|_{\zeta=\eta} - i\sqrt{E}\,\dz\uz(\pm M)|_{\zeta=\eta} - \frac{i}{2\sqrt{E}}\,\uz(\pm M)|_{\zeta=\eta}, \\
    &\dw\Theta^\pm(\eta) = \dw\uz'(\pm M)|_{\zeta=\eta} - i\sqrt{E}\,\dw\uz(\pm M)|_{\zeta=\eta}.
\end{aligned}
\end{equation}
Suppose that $\Phi(0)\neq0$ so that $\uz$ and $\vz$ are defined by \cref{eq:fundamentals1}. Differentiating the equation for $\uz$ we find that $\dw\uz$ solves
\begin{equation}
\label{eq:dwu1}
    (D_x^2 + V(x) - z)\dw\uz = 0, \quad \dw\uz(0)=0, \quad \dw\uz'(0)=1.
\end{equation}
Instead, if $\Phi(0)=0$ we have that $\uz$ and $\vz$ are defined by \cref{eq:fundamentals2} and $\dw\uz$ solves
\begin{equation}
\label{eq:dwu2}
    (D_x^2 + V(x) - z)\dw\uz = 0, \quad \dw\uz(0)=-1, \quad \dw\uz'(0)=0.
\end{equation}
In either case, by virtue of the convention we chose for $\uz$ and $\vz$ in \cref{eq:fundamentals1,eq:fundamentals2}, we have the unique solution
\begin{equation}
\label{eq:dwusolution}
    \dw\uz(x) = \frac{w}{1+w^2}\uz(x) + \vz(x).
\end{equation}

Furthermore, differentiating the equation that $\uz$ solves in $z$, we see that $\dz\uz$ satisfies
\begin{equation}
\label{eq:dzu}
    (D_x^2 + V(x) -z)\dz\uz = \uz, \quad \dz\uz(0)=0, \quad \dz\uz'(0)=0.
\end{equation}
Then by variation of parameters, we have that
\begin{equation}
\label{eq:dzuvop}
    \dz\uz(x) = \bigg(\int_0^x \vz(y)\uz(y)\,dy\bigg)\uz(x) - \bigg(\int_0^x \uz^2(y)\,dy\bigg)\vz(x).
\end{equation}
Setting $\zeta=\eta$, by Floquet theory as discussed above and presented in Sections $5$ and $9$ of \cite{LMW22} and the assumption that $\ue(x) = \Phi(x)$ for $|x|\leq M$, recall we have that $\ue$ exponentially decays in $|x|$ and $\ve$ is limited to exponential growth in $|x|$ as in \cref{eq:uedecay}, \cref{eq:vegrowth}.
We can now give a proof of \cref{lem:thetaestimates}.

\subsection{Proof of \cref{lem:thetaestimates}}
\label{sub:pfthetaestimates}

The estimates \cref{eq:firstestimates} follow immediately from substituting \cref{eq:dwusolution,eq:dzuvop} into \cref{eq:thetaderivatives} and applying the estimates \cref{eq:uedecay,eq:vegrowth}.

To prove estimates \cref{eq:secondestimates}, we will need to bound $\dz^2\Theta^\pm(\zeta)$ uniformly in $z$ and $w$ in a ball centered at $\eta$. We will treat the $z$ estimates first. Differentiating $\Theta^\pm(\zeta)$ gives
\begin{equation}
\label{eq:dz2theta}
    \dz^2\Theta^\pm(\zeta) = \dz^2\uz'(\pm M) + \frac{1}{4}iz^{-3/2}\,\uz(\pm M) - iz^{-1/2}\,\dz\uz(\pm M) - iz^{1/2}\,\dz^2\uz(\pm M)
\end{equation}
and hence it suffices to bound $\dz^2\uz'(M)$, $\uz(M)$, $\dz\uz(M)$, and $\dz^2\uz(M)$ uniformly in $z$ for $z$ nearby $E$. We present and prove the following result, which is the key step in proving \cref{eq:secondestimates} of \cref{lem:thetaestimates} and follows similarly to estimates in Section $6$ of \cite{LMW22}:

\begin{lemma}
\label{lem:nbdbounds}
    Let $0 \neq x\in\mathbb R$ be arbitrary, and let $\zeta\in\mathbb C^2$ be such that
    \begin{equation}
    \label{eq:zetanbd}
        |\zeta-\eta| \leq \frac{1}{2|x|C^*e^{k|x|}}
    \end{equation}
    where $k$ is as in \eqref{eq:uedecay}-\eqref{eq:vegrowth}.
    Then there exists a positive constant $C^*>0$ such that the following estimates hold
    \begin{align}
    \label{eq:estimate1}
      &  \max\{ |\uz(x)|,|\uz'(x)|\} \leq 2C^*e^{k|x|}, \\
    \label{eq:estimate2}
       & \max\{ |\dz\uz(x)|,|\dz\uz'(x)|\} \leq 4C^*  e^{k|x|}
    \end{align}
    and
    \begin{equation}
    \label{eq:estimate3}
        \max\{ |\dz^2\uz(x)|,|\dz^2\uz'(x)|\} \leq 8C^* e^{k|x|}.
    \end{equation}
\end{lemma}

\begin{proof}
    We will prove assertion \cref{eq:estimate1}. Since the proofs of \cref{eq:estimate2,eq:estimate3} are so similar, we will omit them. We will use Picard iteration to solve \cref{eq:fundamentals1,eq:fundamentals2} perturbatively around $\zeta=\eta$. First rewrite \cref{eq:fundamentals1,eq:fundamentals2} as a first order system for $\vecu := (\uz(x),\uz'(x))^T$:
    \begin{equation}
    \label{eq:estimate1system}
        \vecu' = H(x)\vecu + \tilde{H}\vecu,
    \end{equation}
    where
    \begin{equation}
    \label{eq:hmatrices1}
        H(x) := \begin{pmatrix} 0&1 \\ V(x)-E & 0 \end{pmatrix}, \quad \tilde{H} := \begin{pmatrix} 0&0 \\ E-z & 0 \end{pmatrix}.
    \end{equation}
    Note that our initial data varies depending on whether $\Phi(0)\neq0$ or $\Phi(0)=0$, but both cases work in the following argument. Then when $\zeta=\eta$ so that $\tilde H=0$, \cref{eq:estimate1system} has the solution
    \begin{equation}
    \label{eq:estimate1solution}
        \vecu(x) = \Ue(x)\vecu(0)
    \end{equation}
    where $\Ue$ is the solution operator defined as
    \begin{equation}
    \label{eq:solutionoperator}
        \Ue(x) = \begin{pmatrix} \ue(x) & \ve(x) \\ \ue'(x) & \ve'(x) \end{pmatrix}.
    \end{equation}
    We now solve \cref{eq:estimate1system} by Picard iteration for $|\zeta-\eta|$ small. Since $H(x)$ is $x$-dependent, we use the two-point propagator for the homogeneous system $\vec y'=H(x)\vec y$, namely
    \begin{equation}
    \label{eq:twopointpropagator}
        \mathcal P_\eta(x,y):=\Ue(x)\Ue(y)^{-1},
    \end{equation}
    so that $\vec y(x)=\mathcal P_\eta(x,y)\vec y(y)$. By Duhamel's principle, \cref{eq:estimate1system} is equivalent to the fixed point equation
    \begin{equation}
    \label{eq:estimate1duhamel}
        \vecu(x) = \mathcal P_\eta(x,0)\vecu(0) + \int_{0}^{x} \mathcal P_\eta(x,y)\tilde H\vecu(y)\,dy,
    \end{equation}
    where for $x<0$ the integral is understood in the oriented sense, $\int_0^x=-\int_x^0$.

    For any $X\geq0$ we define the operator
    \begin{equation}
    \label{eq:estimate1operator}
        T_X\,:\, \vec f(x) \mapsto \mathcal P_\eta(x,0)\vec f(0) + \int_{0}^{x} \mathcal P_\eta(x,y)\tilde H\vec f(y)\,dy,
    \end{equation}
    acting on the Banach space
    \begin{equation}
    \label{eq:estimate1banach}
        B_x := \{ \vec f\in C([-X,X]:\mathbb C^2)\,:\, \vec f(0) = \vecu(0) \},
    \end{equation}
    with the $\sup$ norm. Recalling that $||\tilde H||_\infty$ is bounded by $|z-E|$, we have
    \begin{equation}
    \begin{aligned}
    \label{eq:estimate1contractioncondition}
        ||T_X\vec f(x) - T_X \vec g(x)||_\infty &\leq \int_{0}^{x} \big\|\mathcal P_\eta(x,y)\big\|_\infty\,||\tilde H||_\infty\,||\vec f-\vec g||_\infty\, dy \\ 
        &\leq |x|\,\sup_{0\leq y\leq x}\big\|\mathcal P_\eta(x,y)\big\|_\infty\,|z-E|\,||\vec f-\vec g||_\infty.
    \end{aligned}
    \end{equation}
    For $y$ between $0$ and $x$ (oriented if $x<0$), $x$ and $y$ have the same sign and $|x|\ge |y|$. Writing $\mathcal P_\eta(x,y)=\Ue(x)\Ue(y)^{-1}$ and using the decay/growth bounds \cref{eq:uedecay,eq:vegrowth} together with $\det\Ue\equiv1$, it is clear that there exists a constant $C^*>0$ such that
    \begin{equation}
        \sup_{0\leq y\leq x}\|\mathcal P_\eta(x,y)\|_\infty \le C^*e^{k(|x|-|y|)}\le C^*e^{k|x|}\le C^*e^{kX}.
    \end{equation}
    Hence
    \begin{equation}
        ||T_X\vec f - T_X \vec g||_\infty\le C^*Xe^{kX}|z-E|\,||\vec f-\vec g||_\infty.
    \end{equation}

    It follows that $T_X$ is a contraction as long as 
    \begin{equation}
    \label{eq:estimate1nbd}
        |\zeta-\eta| < \frac{1}{C^*Xe^{kX}}.
    \end{equation}
    In this case we have a formula for $\vecu(x)$:
    \begin{equation}
    \label{eq:estimate1asymptotic}
        \vecu(x) = \lim_{n\to\infty}T^n \vec u_{\zeta,0},
    \end{equation}
    where $\vec u_{\zeta,0}$ denotes the function equal to the constant $\vecu(0)$ on the interval $[-X,X]$. Iterating \cref{eq:estimate1duhamel} yields the Neumann--Duhamel series
    \begin{equation}
    \begin{aligned}
    \label{eq:estimate1expansion1}
        \vecu(x) &= \mathcal P_\eta(x,0)\vecu(0)
        + \int_{0}^{x} \mathcal P_\eta(x,y_1)\tilde H\,\mathcal P_\eta(y_1,0)\vecu(0)\,dy_1 \\
        &\quad+ \int_{0}^{x}\mathcal P_\eta(x,y_1)\tilde H\int_{0}^{y_1}\mathcal P_\eta(y_1,y_2)\tilde H\,\mathcal P_\eta(y_2,0)\vecu(0)\,dy_2\,dy_1 + \dots,
    \end{aligned}
    \end{equation}
    where all integrals are understood in the oriented sense when $x<0$. The second term on the right hand side can be bounded by
    \begin{equation}
    \begin{aligned}
    \label{eq:estimate1expansionbound1}
       & \bigg|\int_{0}^{x} \mathcal P_\eta(x,y)\tilde H\,\mathcal P_\eta(y,0)\vecu(0)\,dy\bigg|  \\
         & \hspace{2cm} \le |x|\,\sup_{|x|,|y|\le |x|}\|\mathcal P_\eta(x,y)\|_\infty\,\sup_{|y|\le |x|}\|\mathcal P_\eta(y,0)\|_\infty\,|z-E|\,\|\vecu(0)\|_\infty \\
         & \hspace{3cm}  \le C^*|x|\,|z-E|\,e^{2k|x|}\,\|\vecu(0)\|_\infty,
    \end{aligned}
    \end{equation}
    using the same uniform propagator bounds as above.
    Bounding the other terms of the right hand side in a similar fashion,
    we have that
    \begin{equation}
    \begin{aligned}
    \label{eq:estimate1expansionbound2}
        |\vecu(x)| & \leq C^*e^{k|x|} + |x|\,|z-E|\,C^*e^{2k|x|} + |x|^2|z-E|^2C^*e^{3k|x|}+\dots \\
        &= C^*e^{k|x|}(1+|x|\,|z-E|C^*e^{k|x|} + |x|^2|z-E|^2C^*e^{2k|x|} + \dots)
    \end{aligned}
    \end{equation}
    and we see that whenever $|x|\,|z-E|C^*e^{k|x|} < 1$ we have that
    \begin{equation}
    \label{eq:estimate1solutionbound1}
        |\vecu(x)| \leq \frac{C^*e^{k|x|}}{1-|x|\,|z-E|C^*e^{k|x|}}.
    \end{equation}
    By taking
    \begin{equation}
    \label{eq:estimate1nbd1}
        |\zeta-\eta| \leq \frac{1}{2C^*Xe^{kX}}.
    \end{equation}
    we have the desired bound
    \begin{equation}
    \label{eq:estimate1bound1}
        |\vecu(x)| \leq 2C^*e^{k|x|}.
    \end{equation}

    To prove the other assertions \cref{eq:estimate2,eq:estimate3}, we take a similar process of rewriting the equations they solve as first order systems, considering contraction conditions for a solution operator on the same Banach space, and proving bounds of the contraction in the same neighborhood \cref{eq:estimate1nbd1}. This completes the proof. 
\end{proof}

We can now prove \cref{eq:secondestimates}. Note for $\Omega_M$, there exists an $M_0>0$ such that for all $M>M_0$, $\zeta\in\Omega_M$ implies that
\begin{equation}
\label{eq:znbd}
    |z-E| \leq \frac{1}{2MC^*e^{kM}}.
\end{equation}
Applying \cref{lem:nbdbounds} to \cref{eq:dz2theta} and using the triangle inequality, we get
\begin{equation}
\label{eq:dz2thetatriangle}
    |\dz^2\Theta(\zeta)| \leq 8C^*e^{kM} + \frac{1}{4}|z|^{-3/2}2C^*e^{kM} + |z|^{-1/2}4C^*e^{kM} + |z|^{1/2}8C^*e^{kM}.
\end{equation}
Remarking that $|z|^{1/2} \leq \sqrt{E} + C|z-E|$, where $C := \sup_{z\in\Omega_M} \frac{d}{dz}|z|^{1/2}$, we have that (perhaps after taking $M_0$ larger)
\begin{equation}
\label{eq:dz2thetaproved}
    |\dz^2\Theta(\zeta)|\leq Ce^{kM}
\end{equation}
for some constant $C>0$.

\subsection{Proof of \cref{lem:detbound}}
We directly set $\zeta=\eta$ and substitute \cref{eq:dwusolution,eq:dzuvop} into \cref{eq:detdef} for a long expression of $\mathcal{N}(\eta)$. Note that the terms depending quadratically on $\ve(x)$ will dominate for $M$ sufficiently large. After consideration of these terms and simplification, we get that
\begin{equation}
\label{eq:detexpansion}
    \mathcal{N}(\eta) = \int_{-M}^M \ue^2(y)\,dy \big(\ve'(M)-i\sqrt{E}\ve(M)\big) \big(\ve'(-M)-i\sqrt{E}\ve(-M)\big) + O(1).
\end{equation}
It follows that $\mathcal{N}(\eta) \neq 0$ as long as $\ve'(\pm M) - i\sqrt{E}\ve(\pm M)\neq0$ but this clearly holds since $\ve$ is purely real.

Recalling that the Wronskian of $\ue$ and $\ve$ is $1$, by the Cauchy-Schwarz inequality we have that 
\begin{equation}
\begin{aligned}
\label{eq:wronskianinequality}
    1 &= |\ue'(\pm M)\ve(\pm M) - \ue(\pm M)\ve(\pm M)| \\
    &\leq ([\ue(\pm M)]^2 + [\ue'(\pm M)]^2)^{1/2} ([\ve(\pm M)]^2 + [\ve'(\pm M)]^2)^{1/2}
\end{aligned}
\end{equation}
and plugging in \cref{eq:uedecay} we get
\begin{equation}
\begin{aligned}
\label{eq:wronskianinequality2}
    ([\ve(\pm M)]^2 + [\ve'(\pm M)]^2)^{1/2} \geq ([\ue(\pm M)]^2 + [\ue'(\pm M)]^2)^{-1/2} \geq Ce^{kM}.
\end{aligned}
\end{equation}
Hence,
\begin{equation}
\begin{aligned}
\label{eq:vlowerbounds}
    ([\ve(\pm M)]^2 + E[\ve'(\pm M)]^2)^{1/2} \geq Ce^{kM}.
\end{aligned}
\end{equation}
Now, considering \cref{eq:detexpansion} again, we see that
\begin{equation}
\begin{aligned}
\label{eq:detlowerbound}
    \mathcal{N}(\eta) &= \int_{-M}^M \ue^2(y)\,dy \big(\ve'(M)-i\sqrt{E}\ve(M)\big) \big(\ve'(-M)-i\sqrt{E}\ve(-M)\big)\\
    &\geq C\,|\ve'(M)-i\sqrt{E}\ve(M)|\,|\ve'(-M)-i\sqrt{E}\ve(-M)| \\
    &\geq Ce^{2kM},
\end{aligned}
\end{equation}
where the first line is up to $O(1)$ corrections. This proves \cref{lem:detbound}.

\section{Proof of \Cref{cor:reflectionasymptotics}}
\label{sec:reflectionasymptoticsproof}

To prove \cref{cor:reflectionasymptotics}, we need to simplify the correction term of \cref{eq:reflectionlocation} given by
\begin{equation}
\label{eq:reflectionasymptoticscorrection}
    \Xi\Theta(\eta) = \frac{1}{\mathcal N(\eta)} \begin{pmatrix} \dztme\Theta^+(\eta) - \dztpe\Theta^-(\eta) \\ \dwtpe\Theta^-(\eta) - \dwtme\Theta^+(\eta) \end{pmatrix}.
\end{equation}
Taking $\mathcal N(\eta)$ as in \cref{eq:detexpansion} and expanding the components of \cref{eq:reflectionasymptoticscorrection}, ignoring terms independent of $\ve$, and writing $\Xi\Theta(\eta) = ((\Xi\Theta(\eta))_1, (\Xi\Theta(\eta))_2)^T$ gives
\begin{equation}
\begin{aligned}
\label{eq:reflectioncorrection1}
    (\Xi\Theta(\eta))_1 &= \Big[\int_{-M}^0 \ue^2(y)\,dy \big(\ve'(-M)-i\sqrt{E}\ve(-M)\big) \big(\ue'(M)-i\sqrt{E}\ue(M)\big) \\ 
        &+ \int_0^M \ue^2(y)\,dy \big(\ve'(M)-i\sqrt(E)\ve(M)\big) \big(\ue'(-M)-i\sqrt(E)\ue(-M)\big)\Big] \\
        &\times \Big[\int_{-M}^M \ue^2(y)\,dy \big(\ve'(M)-i\sqrt{E}\ve(M)\big) \big(\ve'(-M)-i\sqrt{E}\ve(-M)\big)\Big]^{-1} \\ 
        &+ O(e^{-4kM})
\end{aligned}
\end{equation}
and
\begin{equation}
\begin{aligned}
\label{eq:reflectioncorrection2}
    (\Xi\Theta(\eta))_2 &= \Big[\big(\ve'(M)-i\sqrt{E}\ve(M)\big) \big(\ue'(-M)-i\sqrt{E}\ue(-M)\big) \\
        &- \big(\ve'(-M)-i\sqrt{E}\ve(-M)\big) \big(\ue'(M)-i\sqrt{E}\ue(M)\big)\Big] \\
        &\times \Big[\int_{-M}^M \ue^2(y)\,dy \big(\ve'(M)-i\sqrt{E}\ve(M)\big) \big(\ve'(-M)-i\sqrt{E}\ve(-M)\big)\Big]^{-1} \\ 
        &+ O(e^{-4kM}).
\end{aligned}
\end{equation}
Hence, we have proved \cref{eq:wylocation,eq:zylocation}, and thus \cref{cor:reflectionasymptotics}.

\section{A Comparison of State Locations}
\label{sec:comparison}

A comparison of the precise location of $\zy$ and the known scattering resonance $\zx$ in the same neighborhood amounts to expansions of the formulae presented in \cref{cor:reflectionasymptotics} and the asymptotic expansion of the resonance location given by \cite{LMW22}. We state the main comparison of locations and prove them:

\begin{lemma}
\label{lem:locationcomparison}
    Take the scattering resonance $\zx$ and the point of zero reflection $\zy$ for $H_\text{trunc}$ as found above in $\Omega_M$ defined by \cref{eq:bally}. Take $\ue$, $\ve$ as defined in \cref{eq:fundamentals1,eq:fundamentals2}. Then
    \begin{equation}
    \label{eq:realcomparison}
        \re\zy = \re\zx + O(e^{-4kM})
    \end{equation}
    and
    \begin{equation}
    \label{eq:imaginarycomparison}
        \im\zy = \im\zx + \frac{2\sqrt E}{\Big(\int_{-M}^M \ue^2(y)\,dy \Big)\big([\ve'(-M)]^2+E[\ve(-M)]^2\big)} + O(e^{-4kM}).
    \end{equation}
\end{lemma}

\begin{proof}[Proof of \Cref{lem:locationcomparison}]
    We compare the asymptotic form of $\zy$ given by \cref{eq:zylocation} and the asymptotic form of the scattering resonance, $\zx$, given in \cite{LMW22} (In their notation $\zx = z^*$). This is done in Appendix \ref{app:5pt1}.
\end{proof}

In \Cref{sec:behavior}, we will require estimates on the distance between $\zx$ and $\zy$ to construct a neighborhood that includes $\zy$ and the real line, but not $\zx$. We have the following corollary:

\begin{corollary}
\label{cor:locationdistance}
    Take the assumptions of \cref{thm:zeroreflection} and $\zy$ as found there. Also take the known (\cite{LMW22}) scattering resonance $\zx$ under these same assumptions. Then
    \begin{equation}
    \label{eq:resrefdistance}
        |\im\zy| < |\im(\zx-\zy)|
    \end{equation}
    whenever $M$ is sufficiently large (at least $M\geq M_0$, but perhaps larger) and
    \begin{equation}
    \label{eq:distancequotient1}
        \Bigg|\frac{[\ve'(M)]^2 + E[\ve(M)]^2}{[\ve'(-M)]^2 + E[\ve(-M)]^2}\Bigg| < 2- e^{-3kM}.
    \end{equation} 
    holds.
\end{corollary}
\begin{proof}
    By \cref{eq:imaginarycomparison} we have that
    \begin{equation}
    \label{eq:imaginarydifference}
        |\im(\zx-\zy)| = \frac{2\sqrt E}{\Big(\int_{-M}^M \ue^2(y)\,dy \Big)\big([\ve'(-M)]^2+E[\ve(-M)]^2\big)} + O(e^{-4kM}).
    \end{equation}
    Simplifying and factoring \cref{eq:zylocation} we also get that
    \begin{align}
    \label{eq:imaginaryrelfection}
        |\im\zy| &= \frac{\sqrt E}{\bigg(\int_{-M}^M\ue^2(y)\,dy\bigg) \big([\ve'(M)]^2 + E[\ve(M)]^2\big)} \, \Bigg|\frac{[\ve'(M)]^2 + E[\ve(M)]^2}{[\ve'(-M)]^2 + E[\ve(-M)]^2}\Bigg| \\
        &+ O(e^{-4kM}). \notag
    \end{align}
  Thus \cref{eq:resrefdistance} holds when \cref{eq:distancequotient1} is satisfied.
\end{proof}

\begin{rmk}
\label{rmk:distancequotientcases}
We do not claim that the condition \eqref{eq:distancequotient1} is sharp. It comes purely from our method of proof. However, it is clear to see that a case where it holds is the set of parity-symmetric periodic potentials that satisfy \cref{assn:regularity} and $V(-x)=V(x)$. Indeed, in this setting we have that the point of zero reflection is exactly real:
\end{rmk}
\begin{theorem}
\label{thm:even-exact}
Assume the hypotheses of \cref{thm:zeroreflection}, and suppose in addition
that $V(x)=V(-x)$ for all $x\in\mathbb{R}$. Let $(\wy,\zy)$ be the unique zero of
$\mathbf\Theta$ in $\Omega_M$ as asserted there. Then
\begin{equation}\label{eq:parity-real-zero}
 \zy \in\mathbb R,\qquad \wy \in i\mathbb R.
\end{equation}
\end{theorem}
\begin{proof}
Since $V$ is even, a bound-state eigenfunction at the simple eigenvalue
$E$ has definite parity.  If it is even, the first normalization in
\cref{eq:fundamentals1} applies and its initial slope is $w_0=0$.
If it is odd, the second normalization in \cref{eq:fundamentals2}
applies and again $w_0=0$.  Hence $\Omega_M$ is centered at $(0,E)$.

Choose the branch of $\sqrt z$ analytic near $E>0$ and positive at $E$;
then $\sqrt{\overline z}=\overline{\sqrt z}$ there.  Write
$u_{w,z}$ for the solution used in the definition of $\Theta^\pm$.
For the even initial-value normalization, uniqueness of the ODE implies
\[
 u_{(-\overline w,\overline z)}(x)
   =\overline{u_{(w,z)}(-x)}.
\]
For the odd normalization the corresponding identity is
\[
 u_{(-\overline w,\overline z)}(x)
   =-\overline{u_{(w,z)}(-x)}.
\]
Indeed, the displayed solutions satisfy the appropriate differential
equation and the specified initial conditions at $x=0$.
In either case, if $u_{(w,z)}$ obeys
\[
 u_{(w,z)}'(\pm M)-i\sqrt z\,u_{(w,z)}(\pm M)=0,
\]
then $u_{-\overline w,\overline z}$ obeys the same pair of boundary
conditions with $z$ replaced by $\overline z$: parity reverses the
spatial derivative, while complex conjugation reverses the sign of $i$.
Thus $\mathbf\Theta(w,z)=0$ implies
$\mathbf\Theta(-\overline w,\overline z)=0$.

The map $(w,z)\mapsto(-\overline w,\overline z)$ preserves $\Omega_M$,
because its center is $(0,E)$ and $E$ is real.  The uniqueness in
\cref{thm:zeroreflection} therefore yields
$(\wy,\zy)=(-\overline{\wy},\overline{\zy})$, proving
\cref{eq:parity-real-zero}.
\end{proof}

\section{Behavior of Reflection and Transmission Near Eigenvalues of $H$}
\label{sec:behavior}
We can now discuss the behavior of the transmission and reflection coefficients near our bound state $E$. We will assume the class of potentials that satisfy \cref{eq:distancequotient1} to complete our analysis.

\begin{theorem}
\label{thm:reflectionbounds}
Let $V(x)$ satisfy \Cref{assn:regularity}, let $R(z)$ be the reflection coefficient defined in \cref{eq:scatteringsolutions}, let $\ue$, $\ve$ as in the proof of \Cref{thm:zeroreflection}. Then whenever \cref{eq:distancequotient1} is true, there exists a ball $\Gamma_M$ with
\begin{equation}
\label{eq:gamma1}
    \Gamma_M = \Big\{ z\in\mathbb{C} : |z-\zy| \leq |\im(\zx-\zy)|\Big\},
\end{equation}
where $R(z)$ is analytic in $\Gamma_M$ and the following estimates hold:
\begin{equation}
\label{eq:reflectionbound1}
    \sup_{z\in\Gamma_M}|R(z)| \leq Ce^{-kM}
\end{equation}
and
\begin{equation}
\label{eq:dreflectiondbound1}
    \sup_{z\in\Gamma_M}|\partial_zR(z)| \leq Ce^{kM}.
\end{equation}
\end{theorem}

\begin{proof}
As discussed in for instance \cite{TZ01,DyatlovZworski}, we see that $\zx$ is the zero of $\hat{X}$ and $\zy$ is the zero of $\hat{Y}$, which are two naturally defined functions that arise when considering a set of intertwining operators related to $H_\text{trunc}$. We also have that the reflection coefficient can be expressed as a quotient of these functions:
\begin{equation}
\label{eq:RtoXY}
    \quad R(z) = \frac{\hat{Y}(z)}{\hat{X}(z)}
\end{equation}
as we have taken $R(z) = R_- (z)$ and considered the incoming wave as coming from $-\infty$.
If \cref{eq:distancequotient1} is satisfied, by \cref{cor:locationdistance} we have that $\zx\notin\Gamma_M$ and $R(z)$ is analytic in $\Gamma_M$ defined by \cref{eq:gamma1}. 

To arrive at the bound \cref{eq:reflectionbound1}, notice by \cref{eq:scatteringsolutions}, we can write the reflection coefficient as
\begin{equation}
\label{eq:reflectionform}
    R(z) = \frac{\Theta_-(z,w)}{-2i\sqrt{z}}e^{-i\sqrt{z}\,M} = \frac{\uz'(-M)-i\sqrt{z}\,\uz(-M)}{-2i\sqrt{z}}e^{-i\sqrt{z}\,M}.
\end{equation}
Differentiating this expression in $z$, and using the variation of parameters relation \cref{eq:dzuvop} we get
\begin{equation}
\label{eq:dzreflectionform}
    \partial_zR(z) = \frac{\Big(\int_0^M\uz(y)\,dy\Big)\big(\vz'(M)-i\sqrt{z}\vz(M)-i\uz(M)\big)}{-2i\sqrt{z}}\,e^{-i\sqrt{z}\,M}.
\end{equation}
Since $R$ is analytic in $\Gamma_M$, we can Taylor expand about $\zy$ and take magnitudes to get
\begin{equation}
\begin{aligned}
\label{eq:reflectionTaylor}
    |R(z)| &= \Big|\sum_{j=0}^\infty (j\,!)^{-1}\,\partial_z^jR(\zy)(z-\zy)^j\Big| \\
    &\leq \sum_{j=0}^\infty \big|\partial_z^jR(\zy)\big|\,|z-\zy|^j \\
    &\leq C\sum_{j=0}^\infty\big|\partial_z^jR(\zy)|e^{-2kjM}.
\end{aligned}
\end{equation}
Bounding this series at $j=0$ using the Lagrange error bound, we then get that 
\begin{equation}
\label{eq:reflectionTaylor2}
    \sup_{z\in\Gamma_M}|R(z)| \leq Ce^{-2kM}\bigg(\sup_{z\in\Gamma_M}\big|\partial_zR(z)\big|\bigg).
\end{equation}
By \cref{eq:dzreflectionform} this inequality becomes
\begin{equation}
\label{eq:reflectionTaylor3}
    \sup_{z\in\Gamma_M}|R(z)| \leq (Ce^{-2kM})\,\sup_{z\in\Gamma_M}\bigg(\frac{1}{2|\sqrt{z}|}\big(|\vz'(M) + |\sqrt{z}||\vz(M)| + |\uz(M)|\big)\bigg).
\end{equation}
 Note that we have estimates equivalent to \Cref{lem:nbdbounds} for $\vz$. Remarking that $\Gamma_M$ is a subset of the $z$-component of $\Omega_M$, we can apply these bounds to \cref{eq:reflectionTaylor3} and arrive at
\begin{equation}
\label{eq:reflectionTaylor4}
    \sup_{z\in\Gamma_M}|R(z)| \leq Ce^{-2kM}e^{kM} = Ce^{-kM}.
\end{equation}
The bound \cref{eq:dreflectiondbound1} comes from the same process of Taylor expansion and Lagrange error-bounding. The proof is similar so we exclude it here. 
\end{proof}

\subsection{A Numerical Example}
\label{subsec:numerical}
To give some additional explanation of \Cref{thm:reflectionbounds,thm:zeroreflection}, we can provide a numerical example. We construct a periodic potential with a defect and truncate either side sufficiently far from the defect. We impose boundary conditions and define an incoming wave on the right-hand side, then use Dirichlet finite difference operators to solve for the scattering solutions (transmission/reflection). In Figures \ref{fig:periodic} and \ref{fig:periodicBandGap}, we consider the potential
\[
V=10+5\cos(4\pi x)+5\tanh(x)\cos(2\pi x),
\]
which in \cite{2018ThickeWatsonLu} was demonstrated to have a gap eigenstate.  This potential converges exponentially as $x \to \pm \infty$ to distinct periodic potentials 
\begin{equation}
V_{{\rm per},\pm} = 10+5\cos(4\pi x)\pm5\cos(2\pi x),
\end{equation}
so it is slightly different than the set of conditions we claimed for $V$ in the sense of Remark \ref{rmk:Vpm}, but it is straightforward to see how our methods extend to such a case.  Continuing on, this potential that is known to give a defect state, upon truncation with $M=10$, we see clearly in Fig. \ref{fig:periodic} an extremely resolved peak in the neighborhood of our defect energy, which is $E\approx 19.77$, as well as on a larger scale in Fig. \ref{fig:periodicBandGap} that the $T/R$ coefficients are converging to describe the bands and gaps of the periodic potential.

\begin{figure}[h]
    \centering
    \includegraphics[width=1\linewidth]{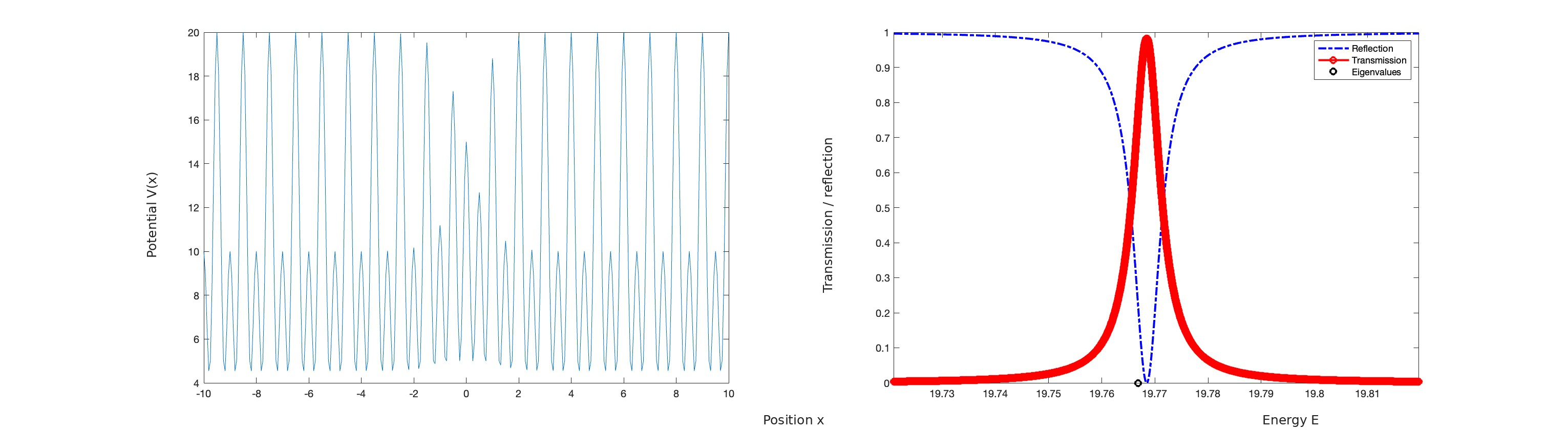}
    \caption{Left: Plot of the periodic potential with defect, $V=10+5\cos(4\pi x)+5\tanh(x)\cos(2\pi x)$, truncated at $M=10$. Right: Plot of the associated transmission-reflection (red-blue) curves around the eigenvalue $E\approx 19.77$}
    \label{fig:periodic}
\end{figure}

\begin{figure}[h]
    \centering
    \includegraphics[width=1\linewidth]{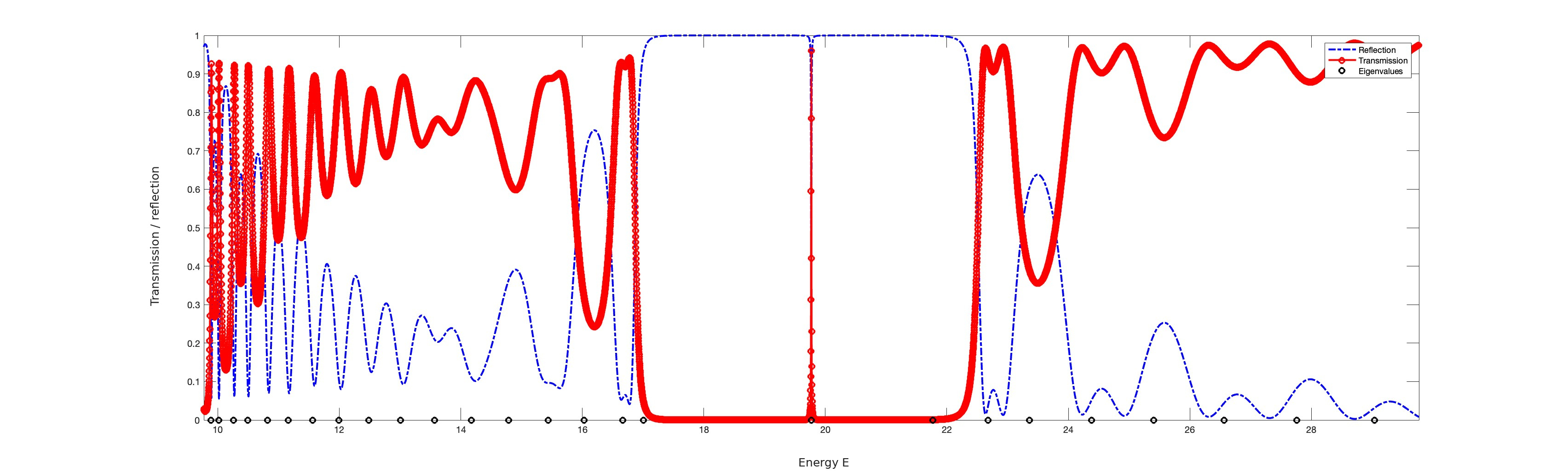}
    \caption{A transmission-reflection plot of the same potential in \cref{fig:periodic}, for a large set of energies. Outside an interval surrounding the bound state with a sharp transmission peak, we see a band-gap structure from the periodic operator's spectrum.}
    \label{fig:periodicBandGap}
\end{figure}

\section{Case Study: The Truncated Harmonic Oscillator}
\label{sec:tho}
As a contrasting
example of reflection zeros generated by a confining well, consider \eqref{eq:h} with $V(x) =
x^2$. This potential does not satisfy \eqref{assn:regularity}: it is unbounded and has no periodic background. Nevertheless, its discrete full-line spectrum and parity symmetry
allow a direct scalar argument, illustrating that our conclusions can persist beyond periodic
media. In this case $H_{\text{trunc}}$ is defined by \cref{eq:htrunc} where
\begin{equation}
\label{eq:thopotential}
    V(x):= \begin{cases} x^2, & |x| \leq M \\ 0, & |x| > M \end{cases}
\end{equation}
for $M>0$. This model is the simple harmonic oscillator truncated at $-M$ and $M$. The structure of this potential will allow us to do a similar analysis to \Cref{sec:reflectionPoint} at the eigenvalues of the simple harmonic oscillator, where we have analytical descriptions of the fundamental solution set.

Note that the full operator \cref{eq:h} with $V(x)=x^2$ has bound state energies at $E_n=2n+1$ for all $n\in\mathbb{N}$, with associated bound states $\Phi_n(x)$ given by
\begin{equation}
\label{eq:thoboundstates}
    (H-E_n)\Phi_n(x) = 0, \qquad \Phi_n(0) = \mathbf{1}_{\text{even}}, \qquad \Phi_n'(0) = \mathbf{1}_{\text{odd}}, \\
\end{equation}
where
\begin{equation}
    \Phi_n(x) = \mathcal{H}_n(x)e^{-\frac{x^2}{2}}.
\end{equation}
Here, $\mathcal{H}_n$ is the n-th standard Hermite polynomial, $\mathbf{1}_{\text{even}}$ is 1 if $n$ is even and $0$ otherwise, and vice-versa for $\mathbf{1}_{\text{odd}}$.

We will prove the existence of zero reflection states very nearby each eigenvalue $E_n$ through arguments analogous to \Cref{sec:reflectionPoint}. When $n$ is even we get $\Phi_n(0)=1$ and $\Phi_n'(0)=0$. In this case we define $\uzn\in C^\infty([-M,M])$ and $\vzn\in C^\infty([-M,M])$ as the solutions of
\begin{equation}
\begin{aligned}
\label{eq:thofundamentals1}
    &(D_x^2+x^2-z_n)\uzn = 0, \qquad \uzn(0)=1, \qquad \uzn'(0) = w_n; \\
    &(D_x^2+x^2-z_n)\vzn = 0, \qquad \vzn(0)=\frac{-w_n}{1+w_n^2}, \qquad \vzn'(0) = \frac{1}{1+w_n^2}
\end{aligned}
\end{equation}
for arbitrary $\zeta_n=(w_n,z_n)\in\mathbb{C}^2$ and let $\eta_n = (0,E_n)$. With these definitions $\uzn$ and $\vzn$ form a fundamental solution set with Wronskian $1$ and $\uen(x) = \Phi_n(x)$ for $|x|\leq M$. When $n$ is odd we get $\Phi_n(0)=0$ and $\Phi_n'(0)=1$. In this case, we define $\uzn$ and $\vzn$ by
\begin{equation}
\begin{aligned}
\label{eq:thofundamentals2}
    &(D_x^2+x^2-z_n)\uzn = 0, \qquad \uzn(0)=-w_n, \qquad \uzn'(0) = 1; \\
    &(D_x^2+x^2-z_n)\vzn = 0, \qquad \vzn(0)=\frac{-1}{1+w_n^2}, \qquad \vzn'(0) = \frac{-w_n}{1+w_n^2}
\end{aligned}
\end{equation}
and we still have $\eta_n=(0,E_n)$. Again, $\uzn$ and $\vzn$ form a fundamental solution set and $\uen(x)=\Phi_n(x)$ for $|x|\leq M$. Unlike the generalized potentials we considered above, we have explicit forms for $\uen$ and $\ven$ on $x\in[-M,M]$:
\begin{equation}
\begin{aligned}
\label{eq:thofundamentalforms}
    &\uen(x) = \mathcal{H}_n(x)e^{-\frac{x^2}{2}}, \\
    \ven(x) = e^{\frac{x^2}{2}}\sum_{k=0}^\infty a_kx^k, &\quad a_{k+2} = -\frac{2(k+n+1)}{(k+2)(k+1)}a_k, \quad \begin{cases} a_0 = \mathbf{1}_{\text{odd}} .\\ a_1 = \mathbf{1}_{\text{even}}. \end{cases}
\end{aligned}
\end{equation}
From this equation, when $n$ is even we get
\begin{equation}
\label{eq:thoevenform}
\ven(x) = e^{\frac{x^2}{2}}\bigg(\sum_{k=0}^\infty\frac{(-2)^kx^{2k+1}}{(2k+1)!}\mathcal{O}(n^k)\bigg),
\end{equation}
and when $n$ is odd we get
\begin{equation}
\label{eq:thooddform}
\ven(x) = e^{\frac{x^2}{2}}\bigg(\sum_{k=0}^\infty\frac{(-2)^kx^{2k}}{(2k)!}\mathcal{O}(n^k)\bigg).
\end{equation}
We define 
\begin{equation}
\label{eq:thoxmap}
    (X_1^\pm(\zeta_n),X_2^\pm(\zeta_n)) := (\uzn(\pm M),\uzn'(\pm M))
\end{equation}
and let
\begin{equation}
\label{eq:thothetamap}
    \Theta^\pm(\zeta_n) := X_2^\pm(\zeta_n) -i\sqrt{z_n}X_1^\pm(\zeta_n).
\end{equation}
Then we have that $\zyn$ is a zero reflection point of $H_{\text{trunc}}$ if and only if
\begin{equation}
\label{eq:thoThetamap}
    \mathbf{\Theta}(\zetayn)=0,\text{ where }\mathbf{\Theta}(\zeta_n) := \begin{pmatrix} \Theta^+(\zeta_n) \\ \Theta^-(\zeta_n) \end{pmatrix},
\end{equation}
where $\zetayn = (\wyn,\zyn)$ for some $\wyn\in\mathbb{C}$. We can now state the main theorem for the truncated harmonic oscillator:

\begin{theorem}
\label{thm:THOmain}
Let $V(x) = x^2$, $n\in\mathbb{N}$, and let $\Phi_n$ be the bound state with eigenvalue $E_n$ as in \cref{eq:thoboundstates}. Then for all $M>0$, $\mathbf{\Theta}(\zeta_n)$ given by \cref{eq:thoThetamap} has a unique zero $\zetayn$ in the ball
\begin{equation}
\label{eq:thoball}
    \Omega_{M,n} := \Big\{ \zeta_n : |\zeta_n-\eta_n| \leq \frac{1}{M^{n+2}}e^{-\frac{M^2}{2}} \Big\},
\end{equation}
where $\zyn\in\mathbb{R}$ and $\wyn\in i\mathbb{R}$.
\end{theorem}

The overall idea of the proof of \Cref{thm:THOmain} is the same as the proof of \Cref{thm:zeroreflection} in \Cref{sec:reflectionPoint}. We construct a map so that a fixed point of the map is a point of zero reflection. We prove this map is a contraction and thus has a unique fixed point in $\Omega_{M,n}$. Define $\Psi:\mathbb{C}^2\rightarrow\mathbb{C}^2$ by
\begin{equation}
\label{eq:thopsimap}
    \Psi(\zeta_n) := \zeta_n - \Xi\mathbf{\Theta}(\zeta_n),
\end{equation}
where $\Xi$ is defined by
\begin{equation}
\label{eq:thoximatrix}
    \Xi := \frac{1}{\mathcal{N}(\eta_n)} \begin{pmatrix} \dztmen & -\dztpen \\ -\dwtmen & \dwtpen \end{pmatrix},
\end{equation}
where $\mathcal{N}(\eta_n) := \dwtpen\dztmen - \dztpen\dwtmen$. Since $\mathcal{N}(\eta_n)\neq0$ then $\Xi$ is invertible and hence
\begin{equation}
\label{eq:thopsithetarelation}
    \Psi(\zeta_n)=\zeta_n \Longleftrightarrow \mathbf{\Theta}(\zeta)=0.
\end{equation}
Then it remains to show that $\Psi$ is a contraction in the ball $\Omega_{M,n}$ given by \cref{eq:thoball} and \Cref{thm:THOmain} follows from the Banach fixed point theorem. 

\section{Proof of \Cref{thm:THOmain}}
By a similar argument made in \Cref{sec:reflectionproof}, we present the following lemmas necessary to show that $\Psi$ is a contraction.
\begin{lemma}
\label{lem:thodetbound}
    Let $\Theta^\pm(\zeta_n)$ be as in \cref{eq:thothetamap}. There exists a positive constant $C>0$ such that for all $M>0$,
    \begin{equation}
        \mathcal{N}(\eta_n) := \dwtpen\dztmen - \dztpen\dwtmen
    \end{equation}
    satisfies
    \begin{equation}
    |\mathcal{N}(\eta_n)| \geq CMe^{\frac{M^2}{2}}.
    \end{equation}
\end{lemma}

\begin{lemma}
\label{lem:thothetaestimates}
    Let $\Omega_{M,n}$ be as in \cref{eq:thoball}. Then there exists a positive constant $C>0$ such that for all $M>0$ the following estimates hold
    \begin{equation}
    \label{eq:thofirstestimates}
        |\dz\Theta^\pm(\eta_n)| \leq CM^{n+1}e^{\frac{M^2}{2}}, |\dw\Theta^\pm(\eta_n)| \leq CM^{n+1}e^{\frac{M^2}{2}}
    \end{equation}
    and
    \begin{equation}
    \label{eq:thosecondestimates}
        \sup_{\zeta_n\in\Omega_{M,n}} |\dz^2\Theta^\pm(\zeta_n)|\leq CM^{n+1}e^{\frac{M^2}{2}}, \sup_{\zeta_n\in\Omega_{M,n}} |\dw^2\Theta^\pm(\zeta_n)|\leq CM^{n+1}e^{\frac{M^2}{2}}.
    \end{equation}
\end{lemma}

The proofs of \cref{lem:thodetbound} and \cref{lem:thothetaestimates} are the same as the proofs for \cref{lem:detbound} and \cref{lem:thetaestimates}, except we use the new estimates resulting from \cref{eq:thoevenform,eq:thooddform}.

\begin{proof}[Proof of \cref{lem:thodetbound}]
    Recall that for the harmonic oscillator we have analytical expressions for $\uen(x)$ and $\ven(x)$ in $x\in[-M,M]$ given by \cref{eq:thoevenform,eq:thooddform}. Remarking that the power series in either case is controllable by a constant, we get
    \begin{equation}
    \label{eq:thouestimates}
        |\uen(\pm M)| \leq CM^ne^{-\frac{M^2}{2}}, |\uen'(\pm M)|\leq CM^{n+1}e^{-\frac{M^2}{2}}
    \end{equation}
    and
    \begin{equation}
    \label{eq:thovestimates}
        |\ven(\pm M)| \leq Ce^{\frac{M^2}{2}}, |\ven'(\pm M)| \leq CMe^{\frac{M^2}{2}}.
    \end{equation}
    Through an analogous consideration of leading order terms in an expansion of $\mathcal{N}(\eta_n)$, which depend quadratically on $\ven$, we arrive at \Cref{lem:thodetbound}.
\end{proof}

\begin{proof}[Proof of \Cref{lem:thothetaestimates}]
    From \cref{eq:thouestimates,eq:thovestimates}, by the same variation of parameters argument used in the general case and expansions of $\Theta^\pm$, we get that
    \begin{equation}
    \label{eq:thothetaestimates}
        |\dz\Theta^\pm(\eta_n)| \leq CM^{n+1}e^{\frac{M^2}{2}}, |\dw\Theta^\pm(\eta_n)| \leq CM^{n+1}e^{\frac{M^2}{2}}.
    \end{equation}
    By the same Picard iteration scheme as presented above, we get the estimates \cref{eq:thosecondestimates} whenever $|\zeta_n-\eta_n| \leq \frac{1}{M^{n+2}}e^{-\frac{M^2}{2}}$.
\end{proof}

Finally, it is clear to see that \Cref{thm:even-exact} applies to the setting of \Cref{thm:THOmain}, since $V(x)=x^2$ is parity-symmetric. 

\subsection{Numerics of the Harmonic Oscialltor}
\label{subsec:THOnumerical}
In a similar fashion to the general periodic potentials presented in \Cref{thm:reflectionbounds,thm:zeroreflection}, we provide a simple numerical example in Figs. \ref{fig:harmonic}, \ref{fig:harmonicSweep}, and \ref{fig:harmonicLargeSweep}. Here, we construct the potential for a related simple harmonic oscillator $V=1+x^2$ and truncate it sufficiently far from zero ($M=4$ in Fig. \ref{fig:harmonic} and $M=3$ in Figs. \ref{fig:harmonicSweep}, \ref{fig:harmonicLargeSweep}).  Peaks in transmission that widen as you move higher up in the spectrum and eventually disappear as you get to high enough energy clearly arise.

\begin{figure}[h]
    \centering
    \includegraphics[width=1\linewidth]{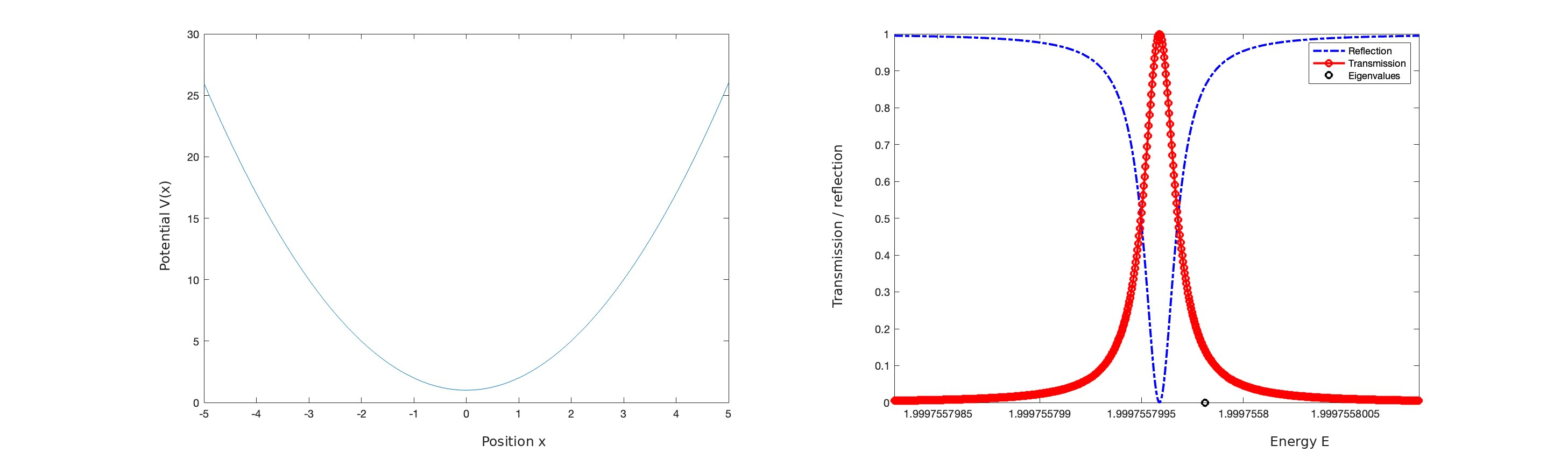}
    \caption{Left: Plot of the simple harmonic oscillator potential ($V=1+x^2$) truncated at $M=5$. Right: The resulting transmission-reflection curves near the first eigenvalue $E=2$ on the right.}
    \label{fig:harmonic}
\end{figure}

\begin{figure}[h]
    \centering
    \includegraphics[width=1\linewidth]{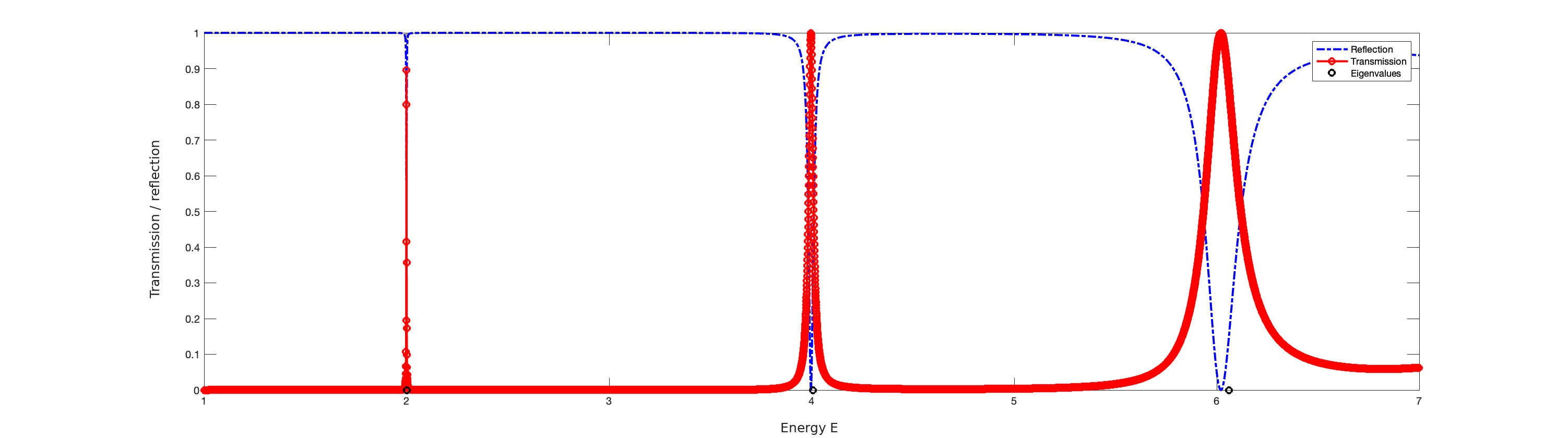}
    \caption{Transmission-reflection plot for the first three eigenvalues of the harmonic oscillator presented in \cref{fig:harmonic} truncated at $M=3$. Transmission peak widths increase as $n$ increases.}
    \label{fig:harmonicSweep}
\end{figure}

\begin{figure}[h]
    \centering
    \includegraphics[width=1\linewidth]{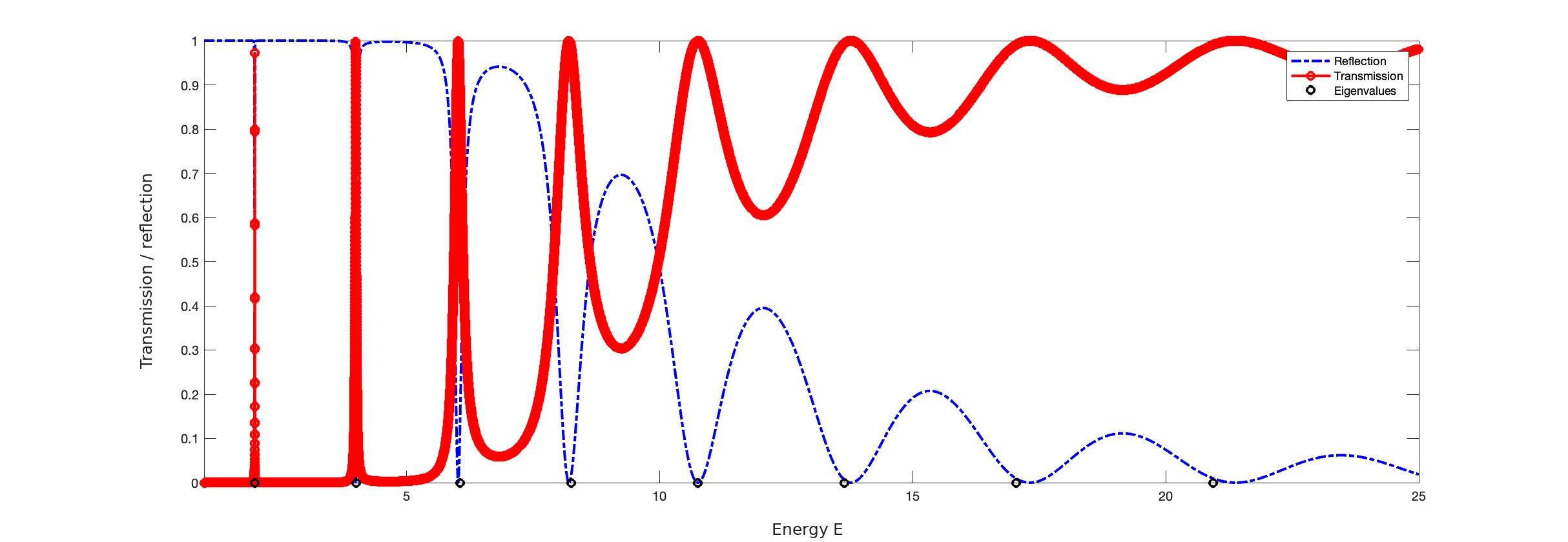}
    \caption{The same harmonic oscillator as in \cref{fig:harmonic,fig:harmonicSweep}, truncated at $M=3$. As energy increases, behavior settles out to pure transmission. Compare to \cref{fig:periodicBandGap}, where the periodic structure leads to bands and gaps in the transmission/reflection plots.}
    \label{fig:harmonicLargeSweep}
\end{figure}

\appendix
\section{Proof of \Cref{lem:locationcomparison}}
\label{app:5pt1}

For convenience and brevity, we will use the following notation to make statements more compact:
\begin{equation}\label{eq:notation}
  u_\eta(\pm M) = u_\pm, \quad u_\eta'(\pm M) = u_\pm', \quad v_\eta(\pm M) = v_\pm, \quad v_\eta'(\pm M) = v_\pm'.
\end{equation}
Expanding the asymptotic form of $\zx$ given in \cite{LMW22}, which is defined as $z^*$, we get
\begin{equation}\label{eq:ResonanceFull}
\begin{aligned}
  & \zx = E- \frac{(v_+'-iE^{1/2}v_+)(u_-'+iE^{1/2}u_-)-(v_-'+iE^{1/2}v_-)(u_+'-iE^{1/2}u_+)}{\big(\int_{-M}^M\ue^2(y)\,dy\big)(v_-'+iE^{1/2}v_-)(v_+'-iE^{1/2}v_+)} \\
  &\qquad+ O(e^{-4kM}) \\
  &= E- \frac{v_+'u_-'-v_-'u_+'+Ev_+u_--Ev_-u_++iE^{1/2}(v_+'u_--v_+u_-'+v_-'u_+-v_-u_+')}{\Big(\int_{-M}^M\ue^2(y)\,dy\Big)(v_-'v_+'+Ev_-v_++iE^{1/2}(v_-v_+'-v_-'v_+))} \\
  &\qquad+ O(e^{-4kM}).
\end{aligned}
\end{equation}
We take the following values for $a,b,c,d\in\mathbb{R}$ given by 
\begin{equation} \label{eq:abcd1}
\begin{aligned}
  a_X &= v_+'u_-'-v_-'u_+'+Ev_+u_--Ev_-u_+, \\
  b_X &= E^{1/2}(v_+'u_--v_+u_-'+v_-'u_+-v_-u_+'), \\
  c_X&= \Big(\int_{-M}^M\ue^2(y)\,dy\Big)(v_-'v_+'+Ev_-v_+) ,\\
  d_X &= \Big(\int_{-M}^M\ue^2(y)\,dy\Big)(E^{1/2}(v_-v_+'-v_-'v_+)),
\end{aligned}
\end{equation}
then we have that $\re(z_X)$ and $\im(\zx)$ are given by
\begin{equation}\label{eq:ReImX}
  \re(\zx) = E - \frac{a_X c_X +b_X d_X}{c_X^2+d_X^2} + O(e^{-4kM}), \qquad \im(\zx) = \frac{a_X d_X -b_X c_X}{c_X^2+d_X^2}+O(e^{-4kM}).
\end{equation}
We first compute the numerator in the fraction term of $\operatorname{Re}(z_X)$ given in \eqref{eq:ReImX}:
\begin{equation}\label{eq:ac1}
\begin{aligned}
  a_X c_X &= \Big(\int_{-M}^M\ue^2(y)\,dy\Big)(v_-'v_+'+Ev_-v_+)(v_+'u_-'-v_-'u_+'+Ev_+u_--Ev_-u_+) \\
  &=\Big(\int_{-M}^M \ue^2(y)\,dy\Big)\Big[(v_+')^2v_-'u_-' -(v_-')^2v_+'u_+' +Ev_-'v_+'v_+u_-' -Ev_-'v_+'v_-u_+ \\
  &\qquad+Ev_-v_+v_+'u_-' -Ev_-v_+v_-'u_+' +E^2v_+^2v_-u_- -E^2v_-^2v_+u_+\Big]
\end{aligned}
\end{equation}
and
\begin{equation}\label{eq:bd1}
\begin{aligned}
  b_X d_X &= \Big(\int_{-M}^M \ue^2(y)\,dy\Big)E(v_-v_+'-v_-'v_+)(v_+'u_- -v_+u_-' +v_-'u_+ -v_-u_+')\\
  &= \Big(\int_{-M}^M \ue^2(y)\,dy\Big)\Big[E(v_+')^2v_-u_- -Ev_+'v_-v_+u_-' +Ev_-'v_+'v_-u_+ -Ev_-^2v_+'u_+' \\
  &\qquad-Ev_-'v_+'v_+u_- +Ev_+^2v_-'u_-' -E(v_-')^2v_+u_+ +Ev_-'v_-v_+u_+'\Big].
\end{aligned}
\end{equation}
Adding \eqref{eq:ac1} and \eqref{eq:bd1}, we have that 
\begin{equation} \label{eq:acplusbd1}
\begin{aligned}
  & a_X c_X+b_X d_X \\
  & = \Big(\int_{-M}^M \ue^2(y)\,dy\Big)\Big[ (v_+')^2v_-'u_-' -(v_-')^2v_+'u_+' +E^2v_+^2v_-u_- -E^2v_-^2v_+u_+ \\
  & \hspace{1cm} \qquad+E(v_+')^2v_-u_- -Ev_-^2v_+'u_+' +Ev_+^2v_-'u_-' -E(v_-')^2v_+u_+\Big]\\
  &= \Big(\int_{-M}^M \ue^2(y)\,dy\Big) \times \\
  &  \hspace{1.5cm} \Big[\big((v_+')^2+Ev_+^2\big)(v_-'u_-'+Ev_-u_-) - \big((v_-')^2+Ev_-^2\big)(v_+'u_+'+Ev_+u_+)\Big].
\end{aligned}
\end{equation}
Now computing the denominator of the fraction term in $\re(\zx)$:
\begin{equation}\label{eq:csquared1}
\begin{aligned}
  c_X^2 &= \Big(\int_{-M}^M \ue^2(y)\,dy\Big)^2(v_-'v_+'+Ev_-v_+)^2 \\
  &= \Big(\int_{-M}^M \ue^2(y)\,dy\Big)^2\big((v_-')^2(v_+')^2 + 2Ev_-'v_+'v_-v_+ + E^2v_-^2v_+^2\big)
\end{aligned}
\end{equation}
and
\begin{equation}\label{eq:dsquared1}
\begin{aligned}
  d_X^2 &= \Big(\int_{-M}^M \ue^2(y)\,dy\Big)^2 E (v_-v_+'-v_-'v_+)^2\\
  &= \Big(\int_{-M}^M \ue^2(y)\,dy\Big)^2\big(Ev_-^2(v_+')^2 - 2Ev_-'v_+'v_-v_+ + E(v_-')^2v_+^2 \big).
\end{aligned}
\end{equation}
Adding \eqref{eq:csquared1} and \eqref{eq:dsquared1}, we have that
\begin{equation}\label{eq:csquaredplusdsquared1}
\begin{aligned}
  c_X^2+d_X^2 &= \Big(\int_{-M}^M \ue^2(y)\,dy\Big)^2\big((v_-')^2(v_+')^2 + Ev_-^2(v_+')^2 + E(v_-')^2v_+^2 + E^2v_-^2v_+^2\big) \\
  &= \Big(\int_{-M}^M \ue^2(y)\,dy\Big)^2\big((v_-')^2+Ev_-^2\big)\big((v_+')^2+Ev_+^2\big).
\end{aligned}
\end{equation}
Dividing \eqref{eq:acplusbd1} by \eqref{eq:csquaredplusdsquared1} and considering \eqref{eq:ReImX}, we have the real part of $\zx$ given by
\begin{align}\label{eq:RealPartX}
  \re(\zx) & = E - \frac{\big((v_+')^2+Ev_+^2\big)(v_-'u_-'+Ev_-u_-) - \big((v_-')^2+Ev_-^2\big)(v_+'u_+'+Ev_+u_+)}{\Big(\int_{-M}^M u_\eta^2(y)\,dy\Big)\big((v_-')^2+Ev_-^2\big)\big((v_+')^2+Ev_+^2\big)} \\
  & + O(e^{-4kM}). \notag
\end{align}
Similarly, we will compute the imaginary part of $\zx$ as given in \eqref{eq:ReImX}. We start by computing the numerator of the fraction term in $\im(\zx)$:
\begin{equation}\label{eq:ad1}
\begin{aligned}
  a_X d_X &= E^{1/2} \Big(\int_{-M}^M \ue^2(y)\,dy\Big) (v_-v_+'-v_-'v_+)(v_+'u_-'-v_-'u_+'+Ev_+u_--Ev_-u_+)\\
  &= E^{1/2} \Big(\int_{-M}^M \ue^2(y)\,dy\Big) \Big[ (v_+')^2v_-u_-' - v_-'v_+'v_-u_+' + Ev_+'v_-v_+u_- - Ev_-^2v_+'u_+ \\
  &\qquad- v_-'v_+'v_+u_-' + (v_-')^2v_+u_+' - Ev_+^2v_-'u_- + Ev_-'v_-v_+u_+ \Big]
\end{aligned}
\end{equation}
and
\begin{equation}\label{eq:bc1}
\begin{aligned}
  b_X c_X &= E^{1/2} \Big(\int_{-M}^M \ue^2(y)\,dy\Big) (v_-'v_+'+Ev_-v_+) (v_+'u_--v_+u_-'+v_-'u_+-v_-u_+') \\
  &= E^{1/2}\Big(\int_{-M}^M \ue^2(y)\,dy\Big) \Big[ (v_+')^2v_-'u_- - v_-'v_+'v_+u_-' + (v_-')^2v_+'u_+ - v_-'v_+'v_-u_+' \\
  &\qquad+ Ev_+'v_-v_+u_- - Ev_+^2v_-u_-' + Ev_-'v_-v_+u_+ - Ev_-^2v_+u_+'\Big].
\end{aligned}
\end{equation}
Subtracting \eqref{eq:bc1} from \eqref{eq:ad1}, we have that
\begin{equation}\label{eq:adminusbc1}
\begin{aligned}
  a_X d_X -b_X c_X &= E^{1/2} \Big(\int_{-M}^M \ue^2(y)\,dy\Big) \Big[(v_+')^2(v_-u_-'-v_-'u_-) + Ev_+^2(v_-u_-'-v_-'u_-) \\
  &\qquad+ (v_-')^2(v_+u_+'-v_+'u_+) + Ev_-^2(v_+u_+'-v_+'u_+)\Big] \\
  &= -E^{1/2} \Big(\int_{-M}^M \ue^2(y)\,dy\Big)\Big[(v_+')^2 + Ev_+^2 + (v_-')^2 + Ev_-^2\Big].
\end{aligned}
\end{equation}
Then dividing \eqref{eq:adminusbc1} by \eqref{eq:csquaredplusdsquared1} and considering \eqref{eq:ReImX}, we have the imaginary part of $\zx$ given by
\begin{equation}\label{eq:ImaginaryPartX}
  \im(\zx) = -\frac{E^{1/2} \big((v_+')^2 + Ev_+^2 + (v_-')^2 + Ev_-^2\big)}{\Big(\int_{-M}^M \ue^2(y)\,dy\Big)\big((v_-')^2+Ev_-^2\big)\big((v_+')^2+Ev_+^2\big)} + O(e^{-4kM})
\end{equation}

We take a similar process of expanding $\zy$ given by \eqref{eq:zylocation} and get
\begin{equation}\label{ZRFull}
\begin{aligned}
  & \zy = E- \frac{(v_+'-iE^{1/2}v_+)(u_-'-iE^{1/2}u_-)-(v_-'-iE^{1/2}v_-)(u_+'-iE^{1/2}u_+)}{\big(\int_{-M}^M\ue^2(y)\,dy\big)(v_-'-iE^{1/2}v_-)(v_+'-iE^{1/2}v_+)} \\
  & \hspace{2cm} + O(e^{-4kM}) \\
  &= E - \frac{v_+'u_-' - Ev_+u_- - v_-'u_+' + Ev_-u_+ + iE^{1/2}(v_-'u_++v_-u_+'-v_+'u_--v_+u_-')}{\Big(\int_{-M}^M \ue^2(y)\,dy\Big)\big(v_-'v_+'-Ev_-v_++iE^{1/2}(-v_-'v_+-v_-v_+')\big)} \\
  & \hspace{2cm} + O(e^{-4kM}).
\end{aligned}
\end{equation}
We similarly define $a_Y,b_Y,c_Y,d_Y\in\mathbb{R}$ to the following values and note that \eqref{eq:ReImX} holds for $\zy$ as well with these new values of $a_Y,b_Y,c_Y,d_Y$.
\begin{equation}\label{eq:abcd2}
\begin{aligned}
  a_Y&= v_+'u_-' - Ev_+u_- - v_-'u_+' + Ev_-u_+,\\
  b_Y&= E^{1/2}(v_-'u_++v_-u_+'-v_+'u_--v_+u_-'),\\
  c_Y&= \Big(\int_{-M}^M \ue^2(y)\,dy\Big)(v_-'v_+'-Ev_-v_+),\\
  d_Y &= \Big(\int_{-M}^M \ue^2(y)\,dy\Big)E^{1/2}(-v_-'v_+-v_-v_+').
\end{aligned}
\end{equation}
We compute all necessary terms as we did for $\zx$:
\begin{equation}\label{eq:ac2}
\begin{aligned}
  a_Y c_Y &= \Big(\int_{-M}^M \ue^2(y)\,dy\Big)(v_-'v_+'-Ev_-v_+)(v_+'u_-' - Ev_+u_- - v_-'u_+' + Ev_-u_+) \\
  &= \Big(\int_{-M}^M \ue^2(y)\,dy\Big)\Big[ (v_+')^2v_-'u_-' - Ev_-'v_+'v_+u_- - (v_-')^2v_+'u_+' + Ev_-'v_+'v_-u_+ \\
  &\qquad- Ev_+'v_-v_+u_-' + E^2v_+^2v_-u_- + Ev_-'u_+'v_-v_+ - E^2v_-^2v_+u_+\Big]
\end{aligned}
\end{equation}
and
\begin{equation}\label{eq:bd2}
\begin{aligned}
  & b_Y d_Y = E\Big(\int_{-M}^M \ue^2(y)\,dy\Big)(v_-'u_+ +v_-u_+' -v_+'u_- - v_+u_-')(-v_-'v_+-v_-v_+')\\
  &= \Big(\int_{-M}^M \ue^2(y)\,dy\Big)\Big[-E(v_-')^2v_+u_+ - Ev_-'v_-v_+u_+' + Ev_-'v_+'v_+u_- + Ev_+^2v_-'u_-' \\
  & \hspace{1cm}- Ev_-'v_+'v_-u_+ - Ev_-^2v_+'u_+' + E(v_+')^2v_-u_- + Ev_+'v_-v_+u_-'\Big].
\end{aligned}
\end{equation}
Thus
\begin{equation}\label{eq:acplusbd2}
\begin{aligned}
 & a_Y c_Y + b_Y d_Y \\
  &= \Big(\int_{-M}^M \ue^2(y)\,dy\Big)\Big[(v_+')^2v_-'u_-' + E^2v_+^2v_-u_- - (v_-')^2v_+'u_+' - E^2v_-^2v_+u_+ \\
  &\qquad- E(v_-')^2v_+u_+ + Ev_+^2v_-'u_-' - Ev_-^2v_+'u_+' + E(v_+')^2v_-u_-\Big]\\
  &= \Big(\int_{-M}^M \ue^2(y)\,dy\Big) \times \\
  &\hspace{1cm} \Big[\big((v_+')^2+Ev_+^2\big)(v_-'u_-'+Ev_-u_-) - \big((v_-')^2+Ev_-^2\big)(v_+'u_+'+Ev_+u_+)\Big].
\end{aligned}
\end{equation}
Also,
\begin{equation}\label{eq:csquared2}
\begin{aligned}
  c_Y^2 &= \Big(\int_{-M}^M \ue^2(y)\,dy\Big)^2 (v_-'v_+'-Ev_-v_+)^2 \\
  &= \Big(\int_{-M}^M \ue^2(y)\,dy\Big)^2 \big( (v_-')^2(v_+')^2 -2Ev_-'v_+'v_-v_+ + E^2v_-^2v_+^2\big)
\end{aligned}
\end{equation}
and
\begin{equation}\label{eq:dsquared2}
\begin{aligned}
  d_Y^2 &= E\Big(\int_{-M}^M \ue^2(y)\,dy\Big)^2 (-v_-'v_+-v_-v_+')^2\\
  &= \Big(\int_{-M}^M \ue^2(y)\,dy\Big)^2\big(E(v_-')^2v_+^2 + 2Ev_-'v_+'v_-v_+ + E(v_+')^2v_-^2\big),
\end{aligned}
\end{equation}
so 
\begin{equation}\label{eq:csquaredplusdsquared2}
\begin{aligned}
  c_Y^2+d_Y^2 &= \Big(\int_{-M}^M \ue^2(y)\,dy\Big)^2\big((v_-')^2(v_+')^2 + Ev_-^2(v_+')^2 + E(v_-')^2v_+^2 + E^2v_-^2v_+^2\big) \\
  &= \Big(\int_{-M}^M \ue^2(y)\,dy\Big)^2\big((v_-')^2+Ev_-^2\big)\big((v_+')^2+Ev_+^2\big).
\end{aligned}
\end{equation}
Then we have the real part of $\zy$ given by
\begin{equation}\label{eq:RealPartY}
\begin{aligned}
  \re(\zy) & = E - \frac{\big((v_+')^2+Ev_+^2\big)(v_-'u_-'+Ev_-u_-) - \big((v_-')^2+Ev_-^2\big)(v_+'u_+'+Ev_+u_+)}{\Big(\int_{-M}^M \ue^2(y)\,dy\Big)\big((v_-')^2+Ev_-^2\big)\big((v_+')^2+Ev_+^2\big)} \\
  & + O(e^{-4kM}).
\end{aligned}
\end{equation}
Furthermore,
\begin{equation}\label{eq:ad2}
\begin{aligned}
  & a_Y d_Y \\
  &= E^{1/2}\Big(\int_{-M}^M \ue^2(y)\,dy\Big)(-v_-v_+'-v_-'v_+)(v_+'u_-'-Ev_+u_--v_-'u_+'+Ev_-u_+)\\
  &= E^{1/2}\Big(\int_{-M}^M \ue^2(y)\,dy\Big)\Big[-v_-'v_+'v_+u_-' + Ev_+^2v_-'u_- + (v_-')^2v_+u_+' - Ev_-'v_+'v_-u_+ \\
  &\qquad- (v_+')^2v_-u_-' + Ev_+'v_-v_+u_- + v_-'v_+'v_-u_+' - Ev_-^2v_+'u_+\Big]
\end{aligned}
\end{equation}
and
\begin{equation}\label{eq:bc2}
\begin{aligned}
  b_Y c_Y &= E^{1/2}\Big(\int_{-M}^M \ue^2(y)\,dy\Big)(v_-'v_+'-Ev_-v_+)(v_-'u_++v_-u_+'-v_+'u_--v_+u_-')\\
  &= E^{1/2}\Big(\int_{-M}^M \ue^2(y)\,dy\Big)\Big[(v_-')^2v_+'u_+ + v_-'v_+'v_-u_+' - (v_+')^2v_-'u_- - v_-'v_+'v_+u_-' \\
  &\qquad- Ev_-'v_-v_+u_+ - Ev_-^2v_+u_+' + Ev_+'v_-v_+u_- + Ev_+^2v_-u_-'\Big],
\end{aligned}
\end{equation}
so
\begin{equation}\label{eq:adminusbc2}
\begin{aligned}
  a_Yd_Y -b_Y c_Y &= E^{1/2}\Big(\int_{-M}^M \ue^2(y)\,dy\Big)\Big[(v_+')^2(v_-'u_--v_-u_-') + (v_-')^2(v_+u_+'-v_+'u_+) \\
  &\qquad+ Ev_+^2(v_-'u_--v_-u_-') + Ev_-^2(v_+u_+'-v_+'u_+)\Big]\\
  &= E^{1/2}\Big(\int_{-M}^M \ue^2(y)\,dy\Big)\big((v_+')^2+Ev_+^2-(v_-')^2-Ev_-^2\big).
\end{aligned}
\end{equation}
Thus, the imaginary part of $\zy$ is given by
\begin{equation}\label{eq:ImaginaryPartY}
\begin{aligned}
  \im(\zy) = \frac{E^{1/2}\big((v_+')^2+Ev_+^2-(v_-')^2-Ev_-^2\big)}{\Big(\int_{-M}^M \ue^2(y)\,dy\Big)\big((v_-')^2+Ev_-^2\big)\big((v_+')^2+Ev_+^2\big)} + O(e^{-4kM}).
\end{aligned}
\end{equation}

Comparing the results of \eqref{eq:RealPartX}, \eqref{eq:RealPartY} and \eqref{eq:ImaginaryPartX}, \eqref{eq:ImaginaryPartY} against each other, we arrive at \eqref{eq:realcomparison}, \eqref{eq:imaginarycomparison} as desired.

\bibliographystyle{siamplain}
\bibliography{references}

\begin{thebibliography}{10}

\bibitem{bandres2018topological}
{\sc M.~A. Bandres, S.~Wittek, G.~Harari, M.~Parto, J.~Ren, M.~Segev, D.~N. Christodoulides, and M.~Khajavikhan}, {\em Topological insulator laser: Experiments}, Science, 359 (2018), p.~eaar4005.

\bibitem{2011BronskiRapti}
{\sc J.~C. Bronski and Z.~Rapti}, {\em Counting defect modes in periodic eigenvalue problems}, SIAM Journal on Mathematical Analysis, 43 (2011), pp.~803--827, \url{https://login.proxy.lib.duke.edu/login?url=https://search.proquest.com/docview/879784374?accountid=10598}.
\newblock Copyright - Copyright Society for Industrial and Applied Mathematics 2011; Document feature - Graphs; Equations; ; Last updated - 2012-06-29.

\bibitem{Cherdantsev2009}
{\sc M.~Cherdantsev}, {\em {Spectral convergence for high-contrast elliptic periodic problems with a defect via homogenization}}, Mathematika, 55 (2009), pp.~29--57, \url{https://doi.org/10.1112/S0025579300000942}.

\bibitem{1986DeiftHempel}
{\sc P.~A. Deift and R.~Hempel}, {\em On the existence of eigenvalues of the schr{\"o}dinger operator $h-\lambda w$ in a gap of $\sigma(h)$}, Comm. Math. Phys., 103 (1986), pp.~461--490, \url{https://projecteuclid.org:443/euclid.cmp/1104114795}.

\bibitem{dobson2013resonances}
{\sc D.~C. Dobson, F.~Santosa, S.~P. Shipman, and M.~I. Weinstein}, {\em Resonances of a potential well with a thick barrier}, SIAM Journal on Applied Mathematics, 73 (2013), pp.~1489--1512.

\bibitem{2019Drouot_2}
{\sc A.~Drouot}, {\em Characterization of edge states in perturbed honeycomb structures}, Pure and Applied Analysis, 1 (2019), pp.~385--445.

\bibitem{2020Drouot}
{\sc A.~Drouot}, {\em Ubiquity of conical points in topological insulators}, arXiv preprint arXiv:2004.07068,  (2020).

\bibitem{2018DrouotFeffermanWeinstein_pre}
{\sc A.~Drouot, C.~L. Fefferman, and M.~I. Weinstein}, {\em Defect modes for dislocated periodic media}.
\newblock \url{arxiv.org/abs/1810.05875}, 2018.

\bibitem{2019DrouotWeinstein}
{\sc A.~Drouot and M.~Weinstein}, {\em Edge states and the valley hall effect}, Advances in Mathematics, 368 (2020), p.~107142.

\bibitem{DVW:15}
{\sc V.~Duch\^ene, I.~Vuki{\'c}evi{\'c}, and M.~Weinstein}, {\em Homogenized description of defect modes in periodic structures with localized defects}, Commun. Math. Sci., 13 (2015), pp.~777--823.

\bibitem{duchene2014scattering}
{\sc V.~Duchene, I.~Vuki{\'c}evi{\'c}, and M.~I. Weinstein}, {\em Scattering and localization properties of highly oscillatory potentials}, Communications on Pure and Applied Mathematics, 67 (2014), pp.~83--128.

\bibitem{duchene2015oscillatory}
{\sc V.~Duch{\^e}ne, I.~Vukicevic, and M.~I. Weinstein}, {\em Oscillatory and localized perturbations of periodic structures and the bifurcation of defect modes}, SIAM Journal on Mathematical Analysis, 47 (2015), pp.~3832--3883.

\bibitem{duchene2011scattering}
{\sc V.~Duch{\^e}ne and M.~I. Weinstein}, {\em Scattering, homogenization, and interface effects for oscillatory potentials with strong singularities}, Multiscale Modeling \& Simulation, 9 (2011), pp.~1017--1063.

\bibitem{DZ19}
{\sc S.~Dyatlov and M.~Zworski}, {\em Mathematical Theory of Scattering Resonances}, American Mathematical Society, 2019.

\bibitem{DyatlovZworski}
{\sc S.~Dyatlov and M.~Zworski}, {\em Mathematical theory of scattering resonances}, Graduate Studies in Mathematics, American Mathematical Society, 2019.

\bibitem{2017FeffermanLee-ThorpWeinstein_2}
{\sc C.~Fefferman, J.~Lee-Thorp, and M.~Weinstein}, {\em Edge states in honeycomb structures}, Annals of PDE, 2 (2016), p.~12.

\bibitem{fefferman_leethorp_weinstein_memoirs}
{\sc C.~Fefferman, J.~Lee-Thorp, and M.~Weinstein}, {\em Topologically protected states in one-dimensional systems}, vol.~247 of Memoirs of the American Mathematical Society, American Mathematical Society, 2017.

\bibitem{2018FeffermanWeinstein}
{\sc C.~L. Fefferman and M.~I. Weinstein}, {\em Edge states of continuum {S}chrodinger operators for sharply terminated honeycomb structures}.
\newblock \url{arxiv.org/abs/1810.03497}, 2018.

\bibitem{2001FigotinGoren}
{\sc A.~Figotin and V.~Goren}, {\em Resolvent method for computations of localized defect modes of h-polarization in two-dimensional photonic crystals}, Phys. Rev. E, 64 (2001), p.~056623, \url{https://doi.org/10.1103/PhysRevE.64.056623}, \url{https://link.aps.org/doi/10.1103/PhysRevE.64.056623}.

\bibitem{1997FigotinKlein}
{\sc A.~Figotin and A.~Klein}, {\em Localized classical waves created by defects}, Journal of Statistical Physics, 86 (1997), pp.~165--177, \url{https://doi.org/10.1007/BF02180202}, \url{https://doi.org/10.1007/BF02180202}.

\bibitem{Figotin-Klein:98}
{\sc A.~Figotin and A.~Klein}, {\em Midgap defect modes in dielectric and acoustic media}, SIAM J. Appl. Math., 58 (1998), pp.~1748--1773.

\bibitem{1993GesztesySimon}
{\sc F.~Gesztesy and B.~Simon}, {\em A short proof of {Z}heludev's theorem}, Trans. Amer. Math. Soc., 335 (1993), pp.~329--340.

\bibitem{harari2018topological}
{\sc G.~Harari, M.~A. Bandres, Y.~Lumer, M.~C. Rechtsman, Y.~D. Chong, M.~Khajavikhan, D.~N. Christodoulides, and M.~Segev}, {\em Topological insulator laser: Theory}, Science, 359 (2018), p.~eaar4003.

\bibitem{Hoefer-Weinstein:11}
{\sc M.~Hoefer and M.~Weinstein}, {\em Defect modes and homogenization of periodic {S}chr{\"o}dinger operators}, SIAM J. Mathematical Analysis, 43 (2011), pp.~971--996.

\bibitem{Kamotski2018a}
{\sc I.~V. Kamotski and V.~P. Smyshlyaev}, {\em {Localized Modes Due to Defects in High Contrast Periodic Media Via Two-Scale Homogenization}}, Journal of Mathematical Sciences (United States), 232 (2018), pp.~349--377, \url{https://doi.org/10.1007/s10958-018-3877-y}.

\bibitem{2019Lee-ThorpWeinsteinZhu}
{\sc J.~P. Lee-Thorp, M.~I. Weinstein, and Y.~Zhu}, {\em Elliptic operators with honeycomb symmetry: Dirac points, edge states and applications to photonic graphene}, Archive for Rational Mechanics and Analysis, 232 (2019), pp.~1--63, \url{https://doi.org/10.1007/s00205-018-1315-4}, \url{https://doi.org/10.1007/s00205-018-1315-4}.

\bibitem{LMW22}
{\sc J.~Lu, J.~L. Marzuola, and A.~B. Watson}, {\em Defect resonances of truncated crystal structures}, SIAM Journal on Applied Mathematics, 82 (2022), \url{https://doi.org/10.1137/21M1415601}.

\bibitem{osting2013long}
{\sc B.~Osting and M.~I. Weinstein}, {\em Long-lived scattering resonances and bragg structures}, SIAM Journal on Applied Mathematics, 73 (2013), pp.~827--852.

\bibitem{1976Simon_2}
{\sc B.~Simon}, {\em The bound state of weakly coupled schr{\"o}dinger operators in one and two dimensions}, Annals of Physics, 97 (1976), pp.~279 -- 288, \url{https://doi.org/https://doi.org/10.1016/0003-4916(76)90038-5}, \url{http://www.sciencedirect.com/science/article/pii/0003491676900385}.

\bibitem{2005Soussi}
{\sc S.~Soussi}, {\em Convergence of the supercell method for defect modes calculations in photonic crystals}, SIAM Journal on Numerical Analysis, 43 (2005), pp.~1175--27, \url{https://login.proxy.lib.duke.edu/login?url=https://search.proquest.com/docview/922260043?accountid=10598}.

\bibitem{TZ01}
{\sc S.~H. Tang and M.~Zworski}, {\em Potential scattering on the real line}.
\newblock unpublished notes.

\bibitem{2018ThickeWatsonLu}
{\sc K.~Thicke, A.~B. Watson, and J.~Lu}, {\em {Computing edge states without hard truncation}}, SIAM J. Sci. Comput.,  (in press).

\bibitem{turner2025resonance}
{\sc J.~C. Turner and M.~I. Weinstein}, {\em Resonance-induced nonlinear bound states}, arXiv preprint arXiv:2510.19538,  (2025).

\end{thebibliography}
\end{document}